\documentclass[11pt,reqno]{amsart}
\usepackage{amssymb,hyperref}
\usepackage[abbrev]{amsrefs}

\numberwithin{equation}{section}

\newtheorem{thm}{Theorem}[section]
\newtheorem{lem}[thm]{Lemma}

\newtheorem{defn}[thm]{Definition}
\newtheorem{rmk}[thm]{Remark}

\def\MB{\mathbb}
\def\MC{\mathcal}

\author[S. Hong]{Sunghyun Hong}
\email{shhong7523@kaist.ac.kr}

\author[S. Kwon]{Soonsik Kwon}
\email{soonsikk@kaist.edu}

\address{Department of Mathematical Sciences, Korea Advanced Institute of Science and Technology, 
 291 Daehak-ro, Yuseong-gu,
 Daejeon, 34141, Korea}

\keywords{Coupled KdV system, Hamiltonian system, nonsqueezing property, normal form}
\subjclass[2010]{Primary  35Q53}

\begin{document}

\title[Nonsqueezing property of the coupled KdV type system]{Nonsqueezing property of the coupled KdV type system without Miura transform}
\begin{abstract}
We prove the nonsqueezing property of the coupled Korteweg-de Vries (KdV) type system. Relying on Gromov's nonsqueezing theorem for finite dimensional Hamiltonian systems, the argument is to approximate the solutions to the original infinite dimensional Hamiltonian system by a frequency truncated finite dimensional system, and then the nonsqueezing property is transferred to the infinite dimensional system. This is the argument used by Bourgain \cite{Bourgain:1994tr} for the 1D cubic NLS flow, and Colliander et. al. \cite{Colliander:2005vv} for the KdV flow. One of main ingredients of \cite{Colliander:2005vv} is to use the Miura transform to change the KdV flow to the mKdV flow. In this work, we consider the coupled KdV system for which the Miura transform is not available. Instead of the Miura transform, we use the method of the normal form via the differentiation by parts. Although we present the proof for the coupled KdV flow, the same proof is applicable to the KdV flow, and so we provide an alternative simplified proof for the KdV flow.
\end{abstract}
\thanks{} 

\maketitle

\section{Introduction}
We consider a Hamiltonian dynamics property, the symplectic nonsqueezing, of the coupled KdV type system,
\begin{equation*}
\left\{\begin{array}{ll}\label{eq:Coupled KdV system}\tag{CKdV}
u_t + u_{xxx} + \frac{1}{2}\left(vv\right)_x = 0 \\
v_t + v_{xxx} + \left(uv\right)_x =0 \\
\left. \left(u,v\right) \right|_{t=0}= \left(u_0\left(x\right), v_0\left(x\right)\right), &\left(u_0,v_0\right) \in H^{-1/2}\left(\MB{T}\right) \times H^{-1/2}\left(\MB{T}\right),
\end{array}
\right.
\end{equation*}
where $\left(x,t\right) \in \MB{T} \times \MB{R} = \left[0,2\pi\right) \times \MB{R}$, and $u$, $v$ are real valued functions.  (\ref{eq:Coupled KdV system}) is a special version of the coupled KdV system,
\begin{equation}\label{eq:Generalized Coupled KdV system}
\begin{cases}{} 
u_t + a_{11} u_{xxx} + a_{12}v_{xxx} + b_1 u u_{x} +b_2 uv_{x}+b_3 u_{x}v +b_4 v v_{x}= 0 \\
v_t + a_{21} u_{xxx} + a_{22}v_{xxx} + b_5 u u_{x} +b_6 uv_{x}+b_7 u_{x}v +b_8 v v_{x}= 0,
\end{cases}
\end{equation}
where $u$,$v$ are real valued functions and $A=\left[\begin{matrix} a_{11} & a_{12} \\ a_{21} & a_{22} \end{matrix}\right]$ is self-adjoint and non-singular. By diagonalization of $A$, we can reduce (\ref{eq:Generalized Coupled KdV system}) to
\begin{equation}\label{eq:Diagonalize Coupled KdV system}
\begin{cases}{} 
u_t +  u_{xxx} + b_1 u u_{x} +b_2 uv_{x}+b_3 u_{x}v +b_4 v v_{x}= 0 \\
v_t  + \alpha v_{xxx} + b_5 u u_{x} +b_6 uv_{x}+b_7 u_{x}v +b_8 v v_{x}= 0,
\end{cases}
\end{equation}
where $\alpha \not =0$ on the same domain. There are many examples of this type of systems, such as the Majda-Biello system,
\begin{equation}\label{eq:Majda-Biello system}
\begin{cases}{} 
u_t +  u_{xxx} + \frac{1}{2}\left(vv\right)_{x}= 0 \\
v_t + \alpha v_{xxx} + \left(uv\right)_{x}= 0. 
\end{cases}
\end{equation}
The KdV type equations can be seen as examples of nonlinear dispersive equations or Hamiltonian systems. If one consider them on a compact domain, due to lack of dispersion, it is better to see them as Hamiltonian systems. They have studied by many authors for both periodic and nonperiodic settings. The studies of the well-posedness have done via local smoothing estimates and Bourgain's $X^{s,b}$ analysis \cite{KENIG:1993ts, Bourgain:1993cl, KENIG:1996tq}. In the low regularity below $L^2$ the global well-posedness is obtained by I-method \cite{Colliander:2003td}. The Majda-Biello system \eqref{eq:Majda-Biello system} is an example of the coupled KdV system. As an extension of results of the KdV equations, the Cauchy problem of the Majda-Biello was well studied by Oh \cite{Oh:2009uu, Oh:2009ux, Oh:2009um}. More precisely, the local well-posedness and the almost surely global well-posedness were proved in \cite{Oh:2009ux, Oh:2009um}, respectively. In \cite{Oh:2009uu}, the global well-posedness of (\ref{eq:Majda-Biello system}) in $H^{s}\left(\MB{T}\right) \times H^{s}\left(\MB{T}\right)$ for $s \ge -1/2$, when $\alpha=1$, was proved (moreover, Oh \cite{Oh:2009uu} obtained the global well-posedness of (\ref{eq:Majda-Biello system}) in $H^{s}\left(\MB{T}\right) \times H^{s}\left(\MB{T}\right)$, $s \ge \tilde s$, where $\tilde s = \tilde s\left(\alpha \right) \in \left( \left. 5/7,1 \right] \right.$, and $\alpha$ satisfies the certain Diophantine condition). Note that the phase space should be $H^{-1/2}(\MB{T}) \times H^{-1/2}(\MB{T})$ to consider the solution flow of \eqref{eq:Coupled KdV system} as a symplectic map. Recently, Guo, Simon and Titi \cite{Guo:2013up} proved the unconditional well-posedness of (\ref{eq:Coupled KdV system}) by the differentiation by parts. \\
\\
The purpose of this paper is to show the symplectic nonsqueezing property of the solution flow of \eqref{eq:Coupled KdV system}. The Lebesgue measure is a typical invariant of a symplectic transform. In \cite{Gromov:1985ww}, Gromov discovered another invariant of a symplectic transform which is called Darboux width. Later, Hofer and Zehnder \cite{Hofer:2011vo} developed the theory of the symplectic capacity. Moreover, Kuksin \cite{Kuksin:1995ue} extended the symplectic capacity to Hamiltonian PDEs. 
The main idea of \cite{Kuksin:1995ue} is that one can approximate the solution flow of the given Hamiltonian PDE as a finite dimensional symplectic map on the phase space. 
Concrete examples were studied by Bourgain \cite{Bourgain:1994tr} for the 1D cubic nonlinear Schr\"odinger equation (NLS), and Colliander et al. \cite{Colliander:2005vv} for the KdV equation. Recently, Roum\'egoux \cite{Roumegoux:2010sn} also proved nonsqueezing for the BBM equation. Mendelson \cite{Mendelson:2014vh} proved the nonsqueezing property of the Klein-Gordon equation on $\MB{T}^3$ via a probabilistic approach.\\ 
\\
As mentioned above, Bourgain \cite{Bourgain:1994tr} proved the nonsqueezing property of the 1D cubic NLS on $L^2_x\left(\MB{T}\right)$ space, and the basic strategy in \cite{Bourgain:1994tr} was an approximation by a finite dimensional truncated flow. The main step is to approximate the 1D cubic NLS flow which is the flow of the infinite dimensional Hamiltonian system by a frequency truncated finite dimensional system. Then due to Gromov's nonsqueezing of finite dimensional Hamiltonian systems, we have the nonsqueezing property of the truncated flow, and the result is transferred to the infinite dimensional NLS flow. Note that the truncated flow should be a symplectic map. Thus, the main here is to find a good such frequency truncation. The 1D cubic NLS is turned out to be well-behaved with the frequency truncations. Indeed,  Bourgain used a basic (or a sharp) frequency truncation, and $X^{s,b}$ space to apply this argument. Later, this argument extended by Colliander et al. \cite{Colliander:2005vv} for the KdV flow on its phase space $H^{-1/2}_x\left(\MB{T}\right)$. In \cite{Colliander:2005vv}, there are two additional ingredients. Firstly, it turned out that the sharp frequency truncation is not working efficiently. They provided a counterexample that a sharp truncation does not approximate the original flow. Instead, they use a smooth truncation to resolve this problem. Secondly, they use the Miura transform to change the KdV flow to the mKdV flow. In fact, they proved the approximation by the truncated flow for the mKdV flow, and using the Miura transform and the inverse of it, concluded the approximation for the KdV flow.  \\
The main goal of our work is to show the second ingredient, the Miura transform, is not necessary but can be replaced by more general technique, so called the normal form. In fact, the Miura transform is a special feature of the KdV flow due to integrability, and so it not widely applicable. Indeed, the system \eqref{eq:Coupled KdV system} does not enjoy the Miura transform. Although we presented the proof for (CKdV), the same proof works for the KdV flow, and so we think this provides an alternative simplified proof of the result in \cite{Colliander:2005vv}.\\
\\
The method of the normal form via the differentiation by parts first introduced by Babin, Ilyin and Titi \cite{Babin:2011wb} for the unconditional well-posedness of the KdV equation on $L^2(\MB{T})$, in which the normal form replaces the use of $X^{s,b}$ spaces. This argument is extended to other equations \cite{Kwon-Oh,Guo-Kwon-Oh}. 
Also, Erdogan and Tzirakis \cite{Erdogan:2012gs} used this method with $X^{s,b}$ multilinear estimates to show the global smoothing for the periodic KdV equation. The method of normal form is a way to detect and cancel out the nonresonancy in the nonlinear term. In general, if the characteristic surface is curved, then from the dispersion relation there is no quadratic resonance. Thus, by taking the normal form, the equation is changed to a cubic equation with quadratic boundary terms. See the detail in Section~\ref{Section: proof of different flow}. In \cite{Colliander:2005vv}, we observe that the role of the Miura transform is to change the KdV equation to the mKdV equation to do analysis for trilinear nonlinearity. Thus, we have thought this could be replaced by the normal form method. Note that in this example, both the Miura transform (in \cite{Colliander:2005vv}) and the method of normal form do not utilize full information of \emph{integrability} of the KdV flow. Thus, the method of normal form is more widely applicable to nonintegrable equations. \\
\\
The rest of paper is organized as follows: In Section \ref{Sec2}, we present theorems for the nonsqueezing property. In Section~\ref{Sec3}, we prove lemmas of bilinear and trilinear estimates in $X^{s,b}$ space setting. In Section~\ref{Section: proof of different flow}, we apply the differentiation by parts to the equation (\ref{eq:Coupled KdV system}) and prove key theorems using multilinear estimates.
 

\subsection*{Notations} $ $ \\
For each dyadic number $N$, we denote the Littlewood-Paley projection by
\begin{align*}
\widehat{P_Nu}\left(k\right) := 1_{N \le \left|k\right| <2N}\left(k\right) \hat u_k, \\
\widehat{P_{\le N} u}\left(k \right) := 1_{\left|k\right| \le N}\left(k\right) \hat u_k, \\
\widehat{P_{\ge N} u}\left(k \right) := 1_{\left|k\right| \ge N}\left(k\right) \hat u_k,
\end{align*}
where $1_{\Omega}$ is a characteristic function on $\Omega$. For positive real numbers $x, y$, $x \lesssim y$ denotes $x \le Cy$ for some $C >0$, and $x \sim y$ means $x \lesssim y$ and $y \lesssim x$. We also denote $f = \MC{O}(g)$ by $f \lesssim g$ for positive real valued functions $f$ and $g$. Moreover, $x \ll y$ denotes $x \le cy$ for some small positive constant $c$.\\
\\

\section{Setting and Statement}\label{Sec2}

We consider (\ref{eq:Coupled KdV system}) for simplicity of the argument. Denote  $S_{CKdV}\left(t\right)$ be the nonlinear solution flow of (\ref{eq:Coupled KdV system}). The system (\ref{eq:Coupled KdV system}) enjoys several conservation laws,
\begin{align*}
E_1 = \int_{\MB{T}} u dx, ~ E_2 = \int_{\MB{T}} v dx, \label{eq:mean} \\
M\left(u,v\right) = \int_{\MB{T}} u^2 +v^2 dx,
\end{align*}
and 
\begin{equation}\label{eq:Hamiltonian}
H\left(u,v\right) = \frac{1}{2}\int_{\MB{T}} u_x^2 + v_x^2 - uv^2 dx.
\end{equation}
Especially, (\ref{eq:Hamiltonian}) is the Hamiltonian, i.e., the system (\ref{eq:Coupled KdV system}) has Hamiltonian structure with respect to (\ref{eq:Hamiltonian}). We denote the spatial Fourier transform and the inverse Fourier transform by
\begin{align*}
\MC{F}_x \left(u\right) &= \hat u_k =  \int _{\MB{T}} e^{-ikx}u\left(x\right) dx, \\
 u \left(x \right) &=  \int  e^{ikx}\hat u\left(k\right) dk:=\frac{1}{2 \pi}\sum_{k \in \MB{Z}} \hat u_k e^{ikx}.
\end{align*}
We use the spatial Sobolev space
\begin{equation*}\label{eq:H^s space}
\left\|u\right\|_{H^s_x} = \left\|\left<k\right>^s \hat u\right\|_{L^2_k} := \frac{1}{\left(2\pi\right)^{1/2}} \left(\sum_{k \in \MB{Z}} \left<k\right>^{2s}\left|\hat u\right|^2\right)^{1/2},
\end{equation*}
where $s \in \MB{R}$ and  $\left<k\right> = \left(1+\left|k\right|^2\right)^{1/2}$. Mostly, we work on the mean zero $H^s$ space as follows,
\begin{equation*}
H^s_0 = \left\{u \in H^s : \frac{1}{2\pi} \int_{\MB{T}} u =0 \right\}  ~\text{and}~ \left\|u\right\|_{H_0^s} :=  \left\|\left<k\right>^s \hat u(k)\right\|_{L^2_k}.
\end{equation*}
Since $E_1$ and $E_2$ are preserved quantities, the function space $H^{s}_0 \times H^s_0$ is well-suited for the solution to (\ref{eq:Coupled KdV system}). Note that due to the Galilean transform, one can switch from mean zero solutions to general mean solutions. \\
The equation (CKdV) is a Hamiltonian PDE associated with Hamiltonian (\ref{eq:Hamiltonian}). More precisely, we can write (\ref{eq:Coupled KdV system}) as
\begin{equation*}
\begin{cases}{}
u_t = \nabla_{\omega,u} H\left(u\left(t\right),v\left(t\right)\right) \\
v_t = \nabla_{\omega,v} H\left(u\left(t\right),v\left(t\right)\right), \\
\end{cases}
\end{equation*}
where 
\begin{equation*}
\begin{cases}{}
\omega\left(h,\nabla_{\omega,u}H\left(u\left(t\right),v\left(t\right)\right)\right) := \left. \frac{d}{d\varepsilon}\right|_{\varepsilon=0} H\left(u+\varepsilon h, v\right) \\
\omega\left(h,\nabla_{\omega,v}H\left(u\left(t\right),v\left(t\right)\right)\right) := \left. \frac{d}{d\varepsilon}\right|_{\varepsilon=0} H\left(u, v+\varepsilon h\right),
\end{cases}
\end{equation*}
and 
\begin{equation*}
\omega\left(u,v\right) := \int_{\MB{T}} u \partial^{-1}_x v dx.
\end{equation*}
Thus, we say 
\begin{equation}\label{eq:symplectic form}
\omega_H \left(\left(u,u'\right),\left(v,v'\right)\right)= \omega\left(u,v\right)+ \omega\left(u',v'\right)
\end{equation}
the symplectic form associated with (\ref{eq:Hamiltonian}). Thus, a solution flow of (\ref{eq:Coupled KdV system}) is the Hamiltonian flow in $(H^{-1/2}_0\left(\MB{T}\right) \times H^{-1/2}_0\left(\MB{T}\right), \omega_H)$ corresponding to (\ref{eq:Hamiltonian}). Note that the system is globally well-posed on its phase space $H^{-1/2}\left(\MB{T}\right) \times H^{-1/2}\left(\MB{T}\right)$, and so the solution flow from data $(u_0,v_0)$ to $(u(t),v(t))$ is a symplectic map at any time $t$. \\
\\
Now, we discuss the nonsqueezing theorem. We first recall Gromov's finite dimensional nonsqueezing theorem. 
\begin{thm}[Finite dimensional nonsqueezing theorem]\label{thm:Finite nonsqueezing}
Let $\MC{S}$ be the symplectic map on the $2n$-dimensional phase space. Let $B_R$, $C_{k,r}$ denote a ball with radius $R$, and a cylinder with radius $r$ at $k$-th component, respectively. If
\begin{equation*}
\MC{S}\left(B_R\right) \subseteq C_{k,r},
\end{equation*} 
then $r \ge R$.
\end{thm}
Our strategy is to find a frequency truncated finite dimensional solution flow which is also the Hamiltonian flow, and approximate to the original flow for some sense. Moreover, once we find the finite dimensional approximation, we can transfer the nonsqueezing theorem to the original flow. \\
The first guess is a sharp frequency truncation like \cite{Bourgain:1994tr}, as an approximation of the flow. However, this is not a good approximation for (\ref{eq:Coupled KdV system}) (see Remark \ref{rmk:counter example}). Naturally, we next choose a smooth truncation like \cite{Colliander:2005vv}. More precisely, let $\phi\left(x\right)$ be a smooth even bump function supported to $\left[-N,N\right]$ which equals 1 on $\left[-N/2,N/2\right]$, and $b\left(k\right)$ be the restriction to integers of $\phi\left(x\right)$. We thus consider the smooth truncated system,
\begin{equation}\label{eq:smooth truncated system}\tag{BKdV}
\left\{\ 
\begin{aligned}
\partial_t u + \partial_{xxx} u + \frac{1}{2}B\left(\left(vv\right)_x\right) = 0 \\
\partial_t v + \partial_{xxx} v + B\left(\left(uv\right)_x\right) =0,
\end{aligned}
\right.
\end{equation}
where
\begin{equation*}
\widehat{Bu}\left(k\right) = b\left(k\right) \hat u\left(k\right).
\end{equation*}
Let $S_{BKdV}\left(t\right)$ be the solution flow of (\ref{eq:smooth truncated system}). Clearly, $S_{BKdV}\left(t\right)$ is a finite dimensional solution flow. However, $S_{BKdV}\left(t\right)$ is not a symplectic map, so we need more steps. To construct an appropriate finite dimensional symplectic map with respect to (\ref{eq:symplectic form}), we first consider a modified Hamiltonian. Let $H_N\left(u,v\right)$ be a Hamiltonian which is defined by
\begin{equation*}
H_N\left(u,v\right) := \frac{1}{2} \int u_x^2 + v_x^2 - B\left(u\right)\left(B\left(v\right)\right)^2 dx,
\end{equation*}
on $P_{\le N} H^{-1/2}_0 \left(\MB{T}\right) \times P_{\le N} H^{-1/2}_0 \left(\MB{T}\right)$. Then, we can get the appropriate truncated system by using $H_N\left(u,v\right)$. By the usual gradient with respect to (\ref{eq:symplectic form}),
\begin{equation*}
\begin{aligned}
\left.\frac{d}{d\varepsilon} H_N \left(u+\varepsilon w,v\right) \right|_{\varepsilon=0} &= \int \left(u_x w_x-B\left(w\right)\left(B\left(v\right)\right)^2\right) = \omega\left(w,-u_{xxx}-\frac{1}{2}B\left(\left(B\left(v\right)B\left(v\right)\right)_x\right)\right), \\
\left.\frac{d}{d\varepsilon} H_N \left(u,v+\varepsilon w\right) \right|_{\varepsilon=0} &= \int \left(v_x w_x-B\left(u\right)B\left(v\right)B\left(w\right)\right) = \omega\left(w,-v_{xxx}-B\left(\left(B\left(u\right)B\left(v\right)\right)_x\right)\right).
\end{aligned}
\end{equation*}
Hence, we conclude that the smooth truncated system with respect to $H_N\left(u,v\right)$ is given by
\begin{equation}\label{eq:symplectic truncated system}\tag{FKdV}
\left\{\ 
\begin{aligned}
\partial_t u + \partial_{xxx} u + \frac{1}{2}B\left(\left(B\left(v\right)B\left(v\right)\right)_x\right) = 0, \\
\partial_t v + \partial_{xxx} v + B\left(\left(B\left(u\right)B\left(v\right)\right)_x\right)=0, 
\end{aligned}
\right.
\end{equation}
for initial data $\left(u_0,v_0\right) \in P_{\le N} H^{-1/2}_0 \left(\MB{T}\right) \times P_{\le N} H^{-1/2}_0 \left(\MB{T}\right)$. Let $S_{FKdV}\left(t\right)$ be the solution flow of (\ref{eq:symplectic truncated system}). It is the finite dimensional symplectic map at any time $t$ by the construction and the global well-posedness. We now consider that $S_{FKdV}\left(t\right)$ as a candidate of the good approximation. 
\begin{rmk}\label{rmk:relation of flow}
We observe relation between \eqref{eq:smooth truncated system} and \eqref{eq:symplectic truncated system}. First of all, we apply the operator $B$ to the both sides of \eqref{eq:symplectic truncated system}. Then we can obtain the system,
\begin{equation}\label{eq:double truncated system}\tag{BBKdV}
\left\{\ 
\begin{aligned}
\partial_t u + \partial_{xxx} u + \frac{1}{2}B^2\left(\left(vv\right)_x\right) = 0 \\
\partial_t v + \partial_{xxx} v + B^2\left(\left( uv\right)_x\right)=0, 
\end{aligned}
\right.
\end{equation}
for initial data $\left(Bu_0,Bv_0\right) \in P_{\le N} H^{-1/2}_0 \left(\MB{T}\right) \times P_{\le N} H^{-1/2}_0 \left(\MB{T}
\right)$. We let $S_{BBKdV}\left(t\right)$ be the solution flow of \eqref{eq:double truncated system}, and then by the definition of $S_{FKdV}\left(t\right)$ and $S_{BBKdV}\left(t\right)$,
\begin{equation*}
BS_{FKdV}\left(t\right)\left(u_0\right) = S_{BBKdV}\left(t\right)\left(Bu_0\right).
\end{equation*}
From the definition of $B$, \eqref{eq:double truncated system} is \eqref{eq:smooth truncated system} with $B$ replaced by $B^2$. This relation will be used in the proof of the approximation to the solution flow. 
\end{rmk}
\begin{rmk}\label{rmk:counter example}
The sharp truncation (it uses $P_{\le N}$ instead of $B$ for truncation) turns out to be not a good finite approximation, due to a counterexample by \cite{Colliander:2005vv}. We consider the initial data
\begin{equation*}\label{eq:counter example}
u_0 = v_0 = \sigma^3 \cos\left(k_0x\right) + \sigma N^{1/2} \cos(Nx),
\end{equation*}
and by the similar iterating argument in \cite{Colliander:2005vv} (or consecutive substitution), we can show that the sharp truncated coupled KdV flow does not approximate the original coupled KdV flow. Since the coupled KdV flow with the same initial data can be regarded as the KdV flow, the same counterexample as in \cite{Colliander:2005vv} works in (CKdV).
\end{rmk}
\begin{rmk}
By the same argument in \cite{Colliander:2005vv, Oh:2009uu}, I-method, we can show the global well-posedness of \eqref{eq:symplectic truncated system}. See \cite{Oh:2009uu} for the detail.
\end{rmk}
So far, we have chosen the appropriate truncation and the function spaces. We now define balls and cylinders in the phase space and state the main theorem, the nonsqueezing property of the coupled KdV type system (\ref{eq:Coupled KdV system}).
\begin{defn}
Let $B^{N}_r\left(u_*\right)$ be a finite dimensional ball in $P_{\le N}{H^{-1/2}_0}$ which has radius $r$ and centered at $u_* \in P_{\le N}{H^{-1/2}_0}$. Likewise, $B^{\infty}_r\left(u_*\right)$ is an infinite dimensional ball in ${H^{-1/2}_0}$ which has radius $r$ and centered at $u_* \in {H^{-1/2}_0}$. That is,
\begin{align*}
B^{N}_r\left(u_*\right) := \left\{u \in P_{\le N}{H^{-1/2}_0} : \left\|u-u_*\right\|_{{H^{-1/2}_0}} \le r\right\}, \\
B^{\infty}_r\left(u_*\right) := \left\{u \in {H^{-1/2}_0} : \left\|u-u_*\right\|_{{H^{-1/2}_0}} \le r\right\}.
\end{align*}

For any $k \in \MB{Z}\setminus \left\{0\right\}\left(:=\MB{Z}^*\right)$, we define that $C^{N}_{k,r}\left(z\right)$ is the finite dimensional cylinder in $P_{\le N}{H^{-1/2}_0}$ which has radius $r$ and centered at $z \in \MB{C}$. Likewise, $C^{\infty}_{k,r}\left(z\right)$ is the infinite dimensional cylinder in ${H^{-1/2}_0}$ which has radius $r$ and centered at $z \in \MB{C}$. That is,
\begin{align*}
C^{N}_{k,r}\left(z\right) := \left\{u \in P_{\le N}{H^{-1/2}_0} : \left|k\right|^{-1/2}\left|\hat u_k-z\right| \le r\right\}, \\
C^{\infty}_{k,r}\left(z\right) := \left\{u \in {H^{-1/2}_0} : \left|k\right|^{-1/2}\left|\hat u_k-z\right| \le r\right\}.
\end{align*}
\end{defn}
Now we state our main theorem.
\begin{thm}\label{thm:nonsqueezing of CKdV(main thm)}
Let $k_1, k_2 \in \MB{Z}^*$, $r_1<R_1$, $r_2<R_2$ and $T>0$. In addition, $\left(u_*,v_*\right) \in H_0^{-\frac{1}{2}} \left(\MB{T}\right) \times H_0^{-\frac{1}{2}} \left( \MB{T}\right)$ and $\left(z,w\right) \in \MB{C}^2$. Then
\begin{equation*}
S_{CKdV}\left(T\right)\left(B^{\infty}_{R_1}\left(u_*\right) \times B^{\infty}_{R_2}\left(v_*\right)\right) \not \subseteq C^{\infty}_{k_1,r_1}\left(z\right) \times C^{\infty}_{k_2,r_2}\left(w\right).
\end{equation*}
In other words, there exists a global solution $S_{CKdV}\left(t\right)\left(u_0,v_0\right) \in H^{-\frac{1}{2}}_0 \times H^{-\frac{1}{2}}_0$ to \eqref{eq:Coupled KdV system} such that
\begin{equation*}
\left\|u_0-u_*\right\|_{H^{-1/2}_0} \le R_1, ~ \left|k_1\right|^{-1/2} \left|\left(S_{CKdV}\left(T\right)u_0\right) \hat{}\left(k_1\right)-z\right| > r_1,
\end{equation*}
and
\begin{equation*}
\left\|v_0-v_*\right\|_{H^{-1/2}_0} \le R_2  , ~ \left|k_2\right|^{-1/2} \left|\left(S_{CKdV}\left(T\right)v_0\right) \hat{}\left(k_2\right)-w\right| > r_2,
\end{equation*}
respectively\footnote{Obviously, $S_{CKdV}(t)$ is the flow $\MB{R} \to \MB{R}^2$. However by abuse of notation, let $S_{CKdV}(t)u_0$ and $S_{CKdV}(t)v_0$ denote the first and the second component of $S_{CKdV}(t)(u_0,v_0)$, respectively. Here and in the sequel, we use these notations for all solution flow as well.}.
\end{thm}
Note that no smallness conditions are imposed on $k_i$, $r_i$, $R_i$, $\left(u_*,v_*\right)$ and $\left(z,w\right)$. \\
Our strategy is to construct a truncated solution flow which has the nonsqueezing property, and approximate to the original solution flow. Hence, we need the nonsqueezing theorem associated with the truncated solution flow (\ref{eq:symplectic truncated system}). 

\begin{lem}\label{lem:Nonsqueezing property of finite solution flow}
Let $k_1, k_2 \in \MB{Z}^*$ such that $\left|k_1\right|, \left|k_2\right| \le N$. Let $r_1<R_1$, $r_2<R_2$ and $T>0$. Furthermore, let $\left(u_0,v_0\right) \in P_{\le N}H_0^{-\frac{1}{2}} \left(\MB{T}\right) \times P_{\le N}H_0^{-\frac{1}{2}} \left( \MB{T}\right)$ and $z,w \in \MB{C}$. Then
\begin{equation*}
S_{FKdV}\left(T\right)\left(B^{N}_{R_1}\left(u_0\right) \times B^{N}_{R_2}\left(v_0\right)\right) \not \subseteq C^{N}_{k_1,r_1}\left(z\right) \times C^{N}_{k_2,r_2}\left(w\right)\end{equation*}
\end{lem}
Since $S_{FKdV}\left(T\right)$ is the finite dimensional symplectic map at time $T$, Lemma \ref{lem:Nonsqueezing property of finite solution flow} is a direct consequence of Theorem \ref{thm:Finite nonsqueezing}. Thus, in the rest we prove that two flows, the flow of (\ref{eq:Coupled KdV system}) and (\ref{eq:symplectic truncated system}) are close for sufficiently large $N$. We show that in two steps. Firstly, we prove that solutions agreeing on low frequency data stay close at frequencies $\le N$. Secondly, we show that solutions to the truncated flow stay close to the original flow in low frequencies. The first part is written as follows,

\begin{thm}\label{thm:est. of different data}
Let $T>0$, $\varepsilon>0$, $\left(u_0,v_0\right) \in {H^{{-\frac{1}{2}}}_0} \times  {H^{{-\frac{1}{2}}}_0}$ and $\left(u'_0,v'_0\right) \in {H^{{-\frac{1}{2}}}_0} \times  {H^{{-\frac{1}{2}}}_0}$. There exists a positive integer 
\begin{equation*}
N_0(T, \varepsilon,\left\|u_0\right\|_{H^{-1/2}_0},\left\|u'_0\right\|_{H^{-1/2}_0},\left\|v_0\right\|_{H^{-1/2}_0},\left\|v'_0\right\|_{H^{-1/2}_0})
\end{equation*}
such that for all $N>N_0$, and the data satisfying $P_{\le 2N}\left(u_0,v_0\right) = P_{\le 2N}\left(u'_0,v'_0\right)$,
\begin{equation*}
\begin{aligned}
&\sup_{\left|t\right|\le T}\left\|P_{\le N}\left(S_{CKdV}\left(t\right) u_0-S_{CKdV} \left(t\right)u'_0\right)\right\|_{{H^{-1/2}_0}}  \\
&+\sup_{\left|t\right|\le T} \left\|P_{\le N}\left(S_{CKdV}\left(t\right)v_0-S_{CKdV}\left(t\right)v'_0\right)\right\|_{{H^{-1/2}_0}} \\
&\lesssim \varepsilon.
\end{aligned}
\end{equation*}
\end{thm}

We now compare the solutions to the original flow and the truncated flow. The proof of this case is more involved and form main analysis of this work. However, we introduce a relatively easier way than the former result \cite{Colliander:2005vv}. In this step, we use the method of the normal form to change the flow with trilinear nonlinear terms and bilinear boundary terms. See the detail in Section \ref{Section: proof of different flow}. 
\begin{thm}[Truncation of the flow]\label{thm:est. of truncated flow}
Let $T>0$ and $\varepsilon >0$. There exists a positive integer $N_0(T,\varepsilon,\left\|u_0\right\|_{H^{-1/2}_0},\left\|v_0\right\|_{H^{-1/2}_0})$ such that for all $N> N_0$,
\begin{equation}\label{eq:est. of truncated flow}
\begin{aligned}
&\sup_{\left|t\right|\le T}\left\|P_{\le N^{\frac{1}{2}}}\left(S_{CKdV}\left(t\right)u_0-S_{BKdV}\left(t\right)u_0\right)\right\|_{H_0^{-\frac{1}{2}}} \\
&+ \sup_{\left|t\right|\le T}\left\|P_{\le N^{\frac{1}{2}}}\left(S_{CKdV}\left(t\right)v_0-S_{BKdV}\left(t\right)v_0\right)\right\|_{H_0^{-\frac{1}{2}}} \\
&\lesssim \varepsilon,
\end{aligned}
\end{equation}
where $\left(u_0,v_0\right) \in {H^{{-\frac{1}{2}}}_0} \times {H^{{-\frac{1}{2}}}_0}$ which has the frequency support on $\left[-N,N\right] \times \left[-N,N\right]$.
\end{thm}
Note that we consider $S_{BKdV}\left(t\right)$ instead of $S_{FKdV}\left(t\right)$ in Theorem \ref{thm:est. of truncated flow}. However, it is enough to prove the approximation, because $S_{FKdV}\left(t\right)$ can be represented $S_{BBKdV}\left(t\right)$ by Remark \ref{rmk:relation of flow} and the support of initial data and $b\left(k\right)$. Thus, Theorem \ref{thm:est. of truncated flow} is equivalent to the approximation between $S_{CKdV}\left(t\right)$ and $S_{FKdV}\left(t\right)$. We can now reach the approximation lemma by assuming Theorem \ref{thm:est. of different data} and \ref{thm:est. of truncated flow}.
\begin{lem}[Approximation lemma]\label{lem:Approximation of solution flow}
Let $k_1, k_2 \in \MB{Z}^*$, $A_1, A_2>0$, $T>0$ and $0< \varepsilon \ll 1$. Then there exists a positive integer $N_0\left(k_1,k_2,A_1,A_2,T,\varepsilon\right) \gg \left|k_1\right|, \left|k_2\right|$ such that
\begin{equation*}
\left|k_1\right|^{-\frac{1}{2}}\left|\left({S_{CKdV}\left(T\right)u_0}\right)^{\wedge{}} \left(k_1\right)-\left({S_{FKdV}\left(T\right)u_0}\right)^{\wedge{}}\left(k_1\right) \right| \le \varepsilon,
\end{equation*}
and 
\begin{equation*}
\left|k_2\right|^{-\frac{1}{2}}\left|\left({S_{CKdV}\left(T\right)v_0}\right)^{\wedge{}} \left(k_2\right)-\left({S_{FKdV}\left(T\right)v_0}\right)^{\wedge{}}\left(k_2\right) \right| \le \varepsilon 
\end{equation*}
for $N>N_0\left(k_1,k_2,A_1,A_2,T,\varepsilon\right)$ and all initial data $u_0 \in B^{N}_{A_1}\left(0\right)$ and $v_0 \in B^{N}_{A_2}\left(0\right)$.
\end{lem}
\begin{proof}
We assume that Theorem \ref{thm:est. of different data} and \ref{thm:est. of truncated flow} are true for a while. The following equalities are obtained by support of the operator $B$ and Remark \ref{rmk:relation of flow},
\begin{equation*}
\left({S_{FKdV}\left(t\right)u_0}\right)^{\wedge{}}\left(k_1\right) = \left({BS_{FKdV}\left(t\right)u_0}\right)^{\wedge{}}\left(k_1\right) = \left({S_{BBKdV}\left(Bu_0\right)}\right)^{\wedge{}}\left(k_1\right)
\end{equation*}
for $\left|k_1\right| \ll N_0$. 

The constant $\varepsilon$ in Lemma \ref{lem:Approximation of solution flow} is different from the constant $\varepsilon$ in Theorem \ref{thm:est. of different data} and \ref{thm:est. of truncated flow}, so we let $\varepsilon'$ denote the upper bounds in the theorems. We choose the sufficiently large $N_0\left(k_1,k_2,A_1,A_2,T,\varepsilon\right)$ such that for all $N>N_0$, `$\lesssim \varepsilon'$' can be changed into `$\le \frac{1}{2} \varepsilon$' in Theorem \ref{thm:est. of different data} and \ref{thm:est. of truncated flow}. Thus, we have
\begin{equation*}
\begin{split}
&\left|k_1\right|^{-\frac{1}{2}}\left|\left({S_{CKdV}\left(T\right)u_0}\right)^{\wedge{}} \left(k_1\right)-\left({S_{BBKdV}\left(T\right)Bu_0}\right)^{\wedge{}}\left(k_1\right) \right| \\
&\le \left|k_1\right|^{-\frac{1}{2}}\left|\left({S_{CKdV}\left(T\right)u_0}\right)^{\wedge{}} \left(k_1\right)-\left({S_{CKdV}\left(T\right)Bu_0}\right)^{\wedge{}}\left(k_1\right) \right| \\
&+ \left|k_1\right|^{-\frac{1}{2}}\left|\left({S_{CKdV}\left(T\right)Bu_0}\right)^{\wedge{}} \left(k_1\right)-\left({S_{BBKdV}\left(T\right)Bu_0}\right)^{\wedge{}}\left(k_1\right) \right| \\
&\le \varepsilon,
\end{split}
\end{equation*}
for $N>N_0\left(k_1,k_2,A_1,A_2,T,\varepsilon\right)$ and $\left|k_1\right| \le N^{1/2}$. In the first inequality, we use the triangle inequality. To have the second inequality, we apply Theorem \ref{thm:est. of different data} to the first term, and Theorem \ref{thm:est. of truncated flow} to the second term, respectively. Similarly, we can obtain the estimate with respect to $v_0$.
\end{proof}
\noindent Assuming Lemma \ref{lem:Approximation of solution flow}, we provide a proof of Theorem \ref{thm:nonsqueezing of CKdV(main thm)}.
\begin{proof}[Proof of Theorem \ref{thm:nonsqueezing of CKdV(main thm)}] $ $\\
Let $r_1$, $R_1$, $u_*$, $k_1$, $z$ and $T$ as in the theorem \ref{thm:nonsqueezing of CKdV(main thm)}. Choose $0 < \varepsilon < \frac{R_1-r_1}{2}$, and the ball $B^{\infty}_{R_1}\left(u_*\right) \subset B^{\infty}_{A_1}\left(0\right)$. We also choose $N > N_0\left(k_1, A_1, T, \varepsilon\right)$ so large that
\begin{equation*}
\left\|u_*- P_{\le N} u_*\right\|_{H^{-1/2}_0} \le \varepsilon.
\end{equation*}
By Lemma \ref{lem:Nonsqueezing property of finite solution flow}, we can find initial data $u_0 \in P_{\le N} H^{-\frac{1}{2}}_0\left(\MB{T}\right)$ satisfying
\begin{equation*}
\left\|u_0 - P_{\le N}u_*\right\|_{H^{-1/2}_0} \le R_1-\varepsilon
\end{equation*}
and then at time T,
\begin{equation*}
\left|k_1\right|^{-\frac{1}{2}} \left|\left(S_{FKdV}\left(T\right)u_0\right)^{\wedge{}}\left(k_1\right)-z\right| > r_1+\varepsilon.
\end{equation*}
From the triangle inequality,
\begin{equation*}
\left\|u_0 - u_*\right\|_{H^{-1/2}_0} \le \left\|u_0 - P_{\le N }u_*\right\|_{H^{-1/2}_0} + \left\|P_{\le N }u_*- u_*\right\|_{H^{-1/2}_0} \le R_1 - \varepsilon + \varepsilon = R_1.
\end{equation*}
We thus have the claim by the triangle inequality and Lemma \ref{lem:Approximation of solution flow},
\begin{equation*}
\begin{split}
&\left|k_1\right|^{-\frac{1}{2}} \left|z-\left(S_{CKdV}\left(T\right)u_0\right)^{\wedge{}}\left(k_1\right)\right| \\
&\ge \left|k_1\right|^{-\frac{1}{2}} \left[ \left|z-\left(S_{FKdV}\left(T\right)u_0\right)^{\wedge{}}\left(k_1\right)\right| - \left|\left(S_{FKdV}\left(T\right)u_0\right)^{\wedge{}}\left(k_1\right)-\left(S_{CKdV}\left(T\right)u_0\right)^{\wedge{}}\left(k_1\right)\right|\right] \\
& > r_1+\varepsilon -\varepsilon = r_1.
\end{split}
\end{equation*}
Similarly, we also get the result for $S_{CKdV}\left(t\right)v_0$.
\end{proof}
\noindent Hence, we remain to prove Theorem~\ref{thm:est. of different data} and Theorem~\ref{thm:est. of truncated flow}. In Section ~\ref{Sec3}, we introduce function spaces and prove bilinear and trilinear estimates. The analysis in this part is similar to \cite{Colliander:2004gc, Erdogan:2012gs}. In Section~\ref{Section: proof of different flow}, we use the normal form method via the differentiation by parts, to change the system to cubic system with bilinear boundary terms, and then we apply multilinear estimates to prove Theorem~\ref{thm:est. of different data} and Theorem~\ref{thm:est. of truncated flow}. 

\section{Bi- and Trilinear estimates}\label{Sec3}
In this section, we state and prove bilinear and trilinear estimates that are used in the proof of Theorem~\ref{thm:est. of different data} and \ref{thm:est. of truncated flow}. First of all, we define function spaces to obtain the multilinear estimates. These function spaces are Fourier restriction spaces that are known as the Bourgain space or the $X^{s,b}$-space. We slightly modify them to define $Y^s$ and $Z^s$ spaces for the solutions and nonlinear terms. For fixed $s,b \in \MB{R}$, and a mean-zero function $u\left(x,t\right)$ on $\MB{T} \times \MB{R}$, recall  
\begin{equation*}
\left\|u\right\|_{X^{s,b}}:=\left\|\left<k\right>^s \left<\tau-k^3\right>^b \MC{F} \left(u\right) \left(k, \tau \right)\right\|_{L^2_k L^2_{\tau}},
\end{equation*}
where $\MC{F}$ is the space-time Fourier transform,
\begin{equation*}
\MC{F}\left(u\right)\left(k,\tau\right) = \tilde u\left(k,\tau\right)= \int_{\MB{T} \times \MB{R}} e^{- i  \left(xk+t \tau\right)} u\left(x,t\right) dx dt.
\end{equation*}
However, $X^{s,b}$-space barely fails to control the $L^{\infty}_t H^s_x$ norm on $ \MB{T} \times \MB{R}$. Hence, we use slightly smaller spaces by adding an additional norm,
\begin{equation*}
\left\|u\right\|_{Y^s}:=\left\|u\right\|_{X^{s,1/2}}+\left\|\left<k\right>^s \MC{F}\left(u\right)\right\|_{L^2_k L^1_{\tau}},
\end{equation*}
and the space for nonlinear terms would be  
\begin{equation*}
\left\|u\right\|_{Z^s}:=\left\|u\right\|_{X^{s,-1/2}}+\left\|\frac{\left<k\right>^s}{\left<\tau-k^3\right>} \MC{F}\left(u\right)\right\|_{L^2_k L^1_{\tau}}.
\end{equation*}
Then, we have embeddings as follows: 
\begin{equation}\label{eq:space embedding}
\begin{aligned}
Y^s \subseteq C_tH^s_x \subseteq L^{\infty}_t H^s_x, \\
L^{\infty}_t H^s_x \subseteq  L^{2}_t H^s_x \subseteq  Z^s
\end{aligned}
\end{equation}
in a compact time interval $[0,T]$. \\
\\
We introduce bilinear and trilinear terms that will appear in normal form analysis. 
\begin{equation*}\label{eq: bilinear formula}
F_2\left(u,v\right) := \MC{F}^{-1}\left(\int \sum_{\substack{ k_i \in \MB{Z}^*\\k_0+k_1+k_2=0 }} \frac{\tilde u_{k_1}}{k_1} \frac{\tilde v_{k_2}}{k_2} d\Gamma\right),
\end{equation*}
where $\int \cdots \,d\Gamma$ means the integration taken on the hyperplane
\begin{equation*}
\left\{\left(\tau_0,\tau_1,\tau_2\right) \in \MB{R}^3:\tau_0+\tau_1+\tau_2=0\right\}.
\end{equation*}
In the analysis, we have two types of trilinear terms, namely, resonance or nonresonance terms:
\begin{equation}\label{eq:resonant trilinear}
F_{r}\left(u,v,w\right) :=  \MC{F}^{-1}\left(\int \tilde u_{-k} \sum_{\substack{k_i \in \MB{Z}^* \\k_2+k_3=0} }\tilde v_{k_2} \frac{\tilde w_{k_3}}{k_3} d\Gamma\right),
\end{equation}
and
\begin{equation*}
F_{nr}\left(u,v,w\right) := \MC{F}^{-1}\left(\int\sum_{\substack{k_i \in \MB{Z}^* \\ k_0+k_1+k_2+k_3=0 \\ (k_1+k_2)(k_2+k_3)(k_3+k_1) \not =0}} \tilde u_{k_1} \tilde v_{k_2} \frac{\tilde w_{k_3}}{k_3} d\Gamma\right),
\end{equation*}
the integral $ \int \cdots \, d\Gamma $ is taken on the set
\begin{equation*}
\left\{\left(\tau_0,\tau_1,\tau_2,\tau_3\right) \in \MB{R}^4:\tau_0+\tau_1+\tau_2+\tau_3=0\right\}.
\end{equation*}
For dyadic numbers $ N_i $, we assume $ N_i \sim \left|k_i\right|$. We denote by $n_1, n_2, n_3, n_4 \in \mathbb Z $ frequencies in order, i.e., 
\begin{equation*}
{|n_{1}|} \ge {|n_{2}|} \ge {|n_{3}|} \ge {|n_{4}|} \quad \text{and} \quad  \{n_1,n_2,n_3,n_4\} =\{k_0,k_1,k_2,k_3\}.
\end{equation*}
Similarly, in the case of three frequencies, let $n_1$, $n_2$ and $n_3$ be defined to be the maximum, median and minimum of $k_0$, $k_1$ and $k_2$, respectively. Namely, 
\begin{equation*}
{n_{1}} \ge {n_{2}} \ge {n_{3}} ~\text{and}~ \{n_1,n_2,n_3\} =\{k_0,k_1,k_2\}.
\end{equation*} 
\subsection{Bilinear estimate}
\begin{lem}\label{lem:bi est.}
Let $u,v \in Y^{-1/2}$. Then
\begin{equation}\label{eq: bi est.}
\left\|\partial_x^{-1}u\partial_x^{-1}v\right\|_{Y^{-1/2}} \lesssim \left\|u\right\|_{Y^{-1/2}} \left\|v\right\|_{Y^{-1/2}}. 
\end{equation}
\end{lem}
\begin{proof}
The $X^{-1/2,1/2}$ part of \eqref{eq: bi est.} is a variant of 
\begin{equation}\label{eq:X^s,1/2 est.}
\left\|uv\right\|_{X^{-1/2,1/2}} \lesssim \left\|u\right\|_{Y^{1/2}} \left\|v\right\|_{Y^{1/2}},
\end{equation}
which was proved in Section 4 of \cite{Colliander:2004gc}, so it is done. To prove the $L^2_kL^1_{\tau}$ part, it is enough to show the estimate
\begin{equation*}
\left\|uv\right\|_{L^2_x L^1_{\tau}} \lesssim \left\|u\right\|_{H^{1/2}_x L^1_{\tau}} \left\|v\right\|_{H^{1/2}_x L^1_{\tau}}.
\end{equation*}
It can be obtained by the Young, H\"older, and Sobolev inequalities.
\end{proof}
\subsection{Trilinear estimate} $ $ \\
In this subsection, we prove the trilinear estimate of the following form.
\begin{lem}\label{lem:tri est.}
Let $u,v$ and $w \in Y^{-1/2}$. Then
\begin{equation}\label{eq:tri est.}
\left\|\left(uv  - P_0\left(uv\right)\right)\partial_x^{-1}w \right\|_{Z^{-1/2}} \lesssim \left\|u\right\|_{Y^{-1/2}} \left\|v\right\|_{Y^{-1/2}}\left\|w\right\|_{Y^{-1/2}},
\end{equation}
where $ P_0$ is the Dirichlet projection to zero frequency, i.e., $P_0\left(f\right) := \int_{\MB{T}} f dx$.
\end{lem}
As opposed to the bilinear estimate, the trilinear term contains resonant interactions. We decompose it into resonant part and nonresonant part. We first consider the resonant part.
\begin{lem}\label{lem:tri resonant est}
Let $u,v$ and $w \in Y^{-1/2}$. We have
\begin{equation*}
\left\|F_{r}\left(u,v,w\right)\right\|_{Z^{-1/2}} \lesssim \left\|u\right\|_{Y^{-1/2}}\left\|v\right\|_{Y^{-1/2}}\left\|w\right\|_{Y^{-1/2}}.
\end{equation*}
\end{lem}
\begin{proof}
To prove the lemma, we handle the space variable and the time variable in consecutive order. We first show an estimate for the spatial domain,
\begin{equation*}
\left\|F_{r}\left(u,v,w\right)\right\|_{H_x^{-1/2}} \lesssim \left\|u\right\|_{H_x^{-1/2}}\left\|v\right\|_{H_x^{-1/2}}\left\|w\right\|_{H_x^{-1/2}}.
\end{equation*}
By duality and the Plancherel's, it suffices to prove the estimate
\begin{equation}\label{eq: space domain}
\left|\int z  u dx\right| \left| \int  v  W dx\right| \lesssim  \left\|z\right\|_{H^{1/2}_x}\left\|u\right\|_{H^{-1/2}_x}\left\|v\right\|_{H^{-1/2}_x}\left\|W\right\|_{H^{1/2}_x},
\end{equation}
where $W = \partial_x^{-1} w$. It is deduced by the H\"older inequality. \\
For the time variable, it is obvious that
\begin{equation*}
\left\|\left<\tau-k^3\right>^{-1/2}\MC{F}\left[F_r\left(u,v,w\right)\right]\right\|_{L^2_{k,\tau}} \lesssim \left\|F_r\left(u,v,w\right)\right\|_{L^2_{x,t}},
\end{equation*}
and
\begin{equation*}
\left\|\left<\tau-k^3\right>^{-1}\MC{F}\left(F_r\left(u,v,w\right)\right)\right\|_{L^2_kL^1_{\tau}} \lesssim \left\|F_r\left(u,v,w\right)\right\|_{L^2_{x,t}}.
\end{equation*}
By taking temporal frequency translation $e^{-t \partial_x^3}$, the claim is reduced to 
\begin{equation*}
\left\|F_r\left(u,v,w\right)\right\|_{L^2_{t}H^{-1/2}_x} \lesssim \left\|u\right\|_{H^{1/2}_{t}H^{-1/2}_x}\left\|v\right\|_{H^{1/2}_{t}H^{-1/2}_x}\left\|w\right\|_{H^{1/2}_{t}H^{-1/2}_x},
\end{equation*}
and then the claim follows from \eqref{eq: space domain}, the H\"older inequality, and  the Sobolev inequality.
\end{proof}
Next, we consider the nonresonant case. 
We can prove a slightly stronger estimate for the nonresonant  part.
\begin{lem}\label{lem:local tri. est.}
Let $u,v$ and $w \in Y^{-1/2}$ and $N_0$, $N_1$, $N_2$ and $N_3$ be dyadic numbers. Then
\begin{equation}\label{eq:freq supp nonresonant est.}
\left\|P_{N_0}F_{nr}\left(P_{N_1}u,P_{N_2}v,P_{N_3}w\right)\right\|_{Z^{-1/2}} \lesssim \left(\frac{N_i}{{n_{1}}}\right)^{\sigma} {n_{3}}^{-\sigma} \left\|u\right\|_{Y^{-1/2}}\left\|v\right\|_{Y^{-1/2}}\left\|w\right\|_{Y^{-1/2}},
\end{equation}
for small enough $\sigma>0$ and $i=1 ~\text{or} ~2$.
\end{lem}
A part of proof of the estimate \eqref{eq:freq supp nonresonant est.} relies on 
\begin{equation}\label{eq:L2 multiest.}
\left\|uvw\right\|_{L^2_{x,t}} \lesssim \left\|u\right\|_{X^{0,1/2-\delta}}\left\|v\right\|_{X^{0,1/2-\delta}}\left\|w\right\|_{X^{1/2-\delta,1/2-\delta}},
\end{equation}
for some small $0<\delta\ll 1$. For the proof, see Section 7 in \cite{Colliander:2004gc}.

\begin{proof}
By symmetry, we assume that $i=1$. We first prove the $X^{-1/2,-1/2}$ part, i.e.,
\begin{align*}
\left\|P_{N_0}F_{nr}\left(P_{N_1}u,P_{N_2}v,P_{N_3}w\right)\right\|_{X^{-1/2,-1/2}} \lesssim \left(\frac{N_1}{{n_{1}}}\right)^{\sigma} {n_{3}}^{-\sigma} \left\|u\right\|_{X^{-1/2,1/2}}\left\|v\right\|_{X^{-1/2,1/2}}\left\|w\right\|_{X^{-1/2,1/2}}.
\end{align*}
Without loss of generality, we may assume that all $u_i$ are nonnegative. By duality, it is equivalent to 
\begin{equation}\label{eq:claim of Xsb}
\begin{aligned}
&\left|\iint u_0F_{nr}\left(u_1u_2u_3\right) dxdt\right| \\
&\lesssim \left(\frac{N_1}{{n_{1}}}\right)^{\sigma} {n_{3}}^{-\sigma} \left\|u_0\right\|_{X^{1/2,1/2}}\left\|u_1\right\|_{X^{-1/2,1/2}}\left\|u_2\right\|_{X^{-1/2,1/2}}\left\|U_3\right\|_{X^{1/2,1/2}},
\end{aligned}
\end{equation}
where $U_3 = \partial_x^{-1}u_3$ and $u_i$ has Fourier support on the region $\left|k_i\right| \sim N_i$. The right hand side of (\ref{eq:claim of Xsb}) is comparable to 
\begin{equation}\label{eq:RHS of claim of Xsb}
\left(\frac{N_1}{{n_{1}}}\right)^{\sigma}{n_{3}}^{-\sigma} \frac{\left|N_0N_3\right|^{1/2}}{\left|N_1N_2\right|^{1/2}} \left\|u_0\right\|_{X^{0,1/2}}\left\|u_1\right\|_{X^{0,1/2}}\left\|u_2\right\|_{X^{0,1/2}}\left\|U_3\right\|_{X^{0,1/2}}.
\end{equation}
\begin{lem}\label{lem:est. of coeff}
In the same notation, we have
\begin{equation}\label{eq:est. of coefficient}
\left(\frac{N_1}{{n_{1}}}\right)^{\sigma} \frac{\left|N_0N_3\right|^{1/2}}{\left|N_1N_2\right|^{1/2}} \ge \frac{{n_{3}}^{1/2}{n_{4}}^{1/2}}{{n_{1}}}.
\end{equation}
\end{lem}
\begin{proof}
If $N_1 \sim {n_{1}}$, then we can easily obtain (\ref{eq:est. of coefficient}). Hence we may assume that $N_1 \ll {n_{1}}$. We rewrite (\ref{eq:est. of coefficient}) as 
\begin{equation*}
 \frac{{n_{3}}^{1/2}{n_{4}}^{1/2}}{{n_{1}}} \left(\frac{{n_{1}}}{N_1}\right)^{\sigma} \frac{\left|N_1N_2\right|^{1/2}}{\left|N_0N_3\right|^{1/2}} \le 1.
\end{equation*}
In other words,
\begin{equation*}
\frac{\left|N_1\right|^{1/2-\sigma}}{{n_{1}}^{1/2-\sigma}} \frac{\left|N_2\right|^{1/2}}{{n_{1}}^{1/2}} \frac{{n_{3}}^{1/2}{n_{4}}^{1/2}}{\left|N_0N_3\right|^{1/2}}  \le 1.
\end{equation*}
Each term of the left hand side is smaller than 1, and so we are done.
\end{proof}

From Lemma \ref{lem:est. of coeff}, (\ref{eq:RHS of claim of Xsb}) is bounded below by
\begin{equation}\label{eq:below bound}
\frac{{n_{3}}^{1/2-\sigma}{n_{4}}^{1/2}}{{n_{1}}} \left\|u_0\right\|_{X^{0,1/2}}\left\|u_1\right\|_{X^{0,1/2}}\left\|u_2\right\|_{X^{0,1/2}}\left\|U_3\right\|_{X^{0,1/2}}.
\end{equation}
By a resonance identity,
\begin{equation*}
\sum_{i=0,1,2,3} \left(\tau_i-k^3_i\right) = - \sum_{i=0,1,2,3} k_i^3 = 3\left(k_1+k_2\right)\left(k_2+k_3\right)\left(k_3+k_1\right),
\end{equation*}
and thus,
\begin{equation*}\label{eq:freq identity}
\begin{split}
\sup_{i=0,1,2,3} L_i &\gtrsim \left|k_1+k_2\right|\left|k_2+k_3\right|\left|k_3+k_1\right|,
\end{split}
\end{equation*}
for $L_i = \left<\tau_i-k_i^3\right>$. Due to the symmetry of the functions in (\ref{eq:below bound}), we only consider  that $\displaystyle L_0  = \sup_{i=0,1,2,3} L_i$, and by Lemma 4.4 of \cite{Colliander:2005vv}, we have
\begin{equation*}
L_0 \gtrsim {n_{1}}^2{n_{4}}^{-1}.
\end{equation*}
Therefore, to prove \eqref{eq:claim of Xsb}, it suffices to show that
\begin{equation*}
\left|\sum_{|\bf k|  \in \mathcal{NR} } \int {n_{3}}^{-1/2+\sigma}L_0^{1/2} \tilde u_0 \tilde u_1 \tilde u_2 \tilde U_3 d\Gamma\right| \lesssim \left\|u_0\right\|_{X^{0,1/2}}\left\|u_1\right\|_{X^{0,1/2}}\left\|u_2\right\|_{X^{0,1/2}}\left\|U_3\right\|_{X^{0,1/2}}.
\end{equation*}
Here, as $\int \cdots \, d\Gamma $ is the integral taken over the set $\{ \tau_0 + \tau_1 + \tau_2 +\tau_3 =0\}$ before, and we denote $|{\bf k}| =(k_0,k_1,k_2,k_3) $ and $ \mathcal{NR}= \{(k_0,k_1,k_2,k_3) :  k_0+k_1+k_2+k_3=0, \,(k_1+k_2)(k_2+k_3)(k_3+k_1) \not =0 \}  $. \\
At least one of $k_1$, $k_2$ and $k_3$ is $\MC{O}\left(n_3\right)$. By symmetry, let us suppose that it is $k_3$. Then it is enough to show that
\begin{equation}\label{X part}
\left|\iint u_0 u_1u_2U_3 dxdt\right| \lesssim \left\|u_0\right\|_{L^2_{x,t}}\left\|u_1\right\|_{X^{0,1/2}}\left\|u_2\right\|_{X^{0,1/2}}\left\|U_3\right\|_{X^{1/2-\sigma,1/2}}.
\end{equation}
From the Cauchy-Schwarz inequality and (\ref{eq:L2 multiest.}), we can get \eqref{X part} for sufficiently small $\sigma$.\\
\\
Next, we prove the $L_k^2 L_{\tau}^1$ part. From the H\"older inequality, $X^{-1/2,-1/2}$ part and the interpolation, it is enough to prove the estimate
\begin{equation*}
\left\|\sum_{|\bf k|  \in \mathcal{NR} } \int \frac{1}{\left|k_0\right|^{1/2}L_0^{1-\delta}}\tilde u_1\tilde u_2\tilde U_3d\Gamma\right\|_{L^2_{k_0}L^1_{\tau_0}} \lesssim \left\|u_1\right\|_{X^{-1/2,1/2}}\left\|u_2\right\|_{X^{-1/2,1/2}}\left\|U_3\right\|_{X^{1/2,1/2}},
\end{equation*}
where $U_3 = \partial_x^{-1} u_3$ and some positive constant $\delta \ll 1$. In other words, we will show that
\begin{equation}\label{eq:claim of L2L1}
\left\|\sum_{|\bf k|  \in \mathcal{NR} } \int \frac{\left|k_1k_2\right|^{1/2}}{\left|k_0k_3\right|^{1/2}L_0^{1-\delta}L^{1/2}_1L^{1/2}_2L^{1/2}_3}\tilde u_1\tilde u_2\tilde U_3d\Gamma\right\|_{L^2_{k_0}L^1_{\tau_0}} \lesssim \left\|u_1\right\|_{L^2_{x,t}}\left\|u_2\right\|_{L^2_{x,t}}\left\|U_3\right\|_{L^2_{x,t}}.
\end{equation}
From the H\"older inequality, the left hand side of (\ref{eq:claim of L2L1}) is bounded by
\begin{equation}\label{eq:i case after holder}
\left\|\sum_{|\bf k|  \in \mathcal{NR} } \int \frac{\left|k_1k_2\right|^{1/2}}{\left|k_0k_3\right|^{1/2}L_0^{1/2-2\delta}L^{1/2}_1L^{1/2}_2L^{1/2}_3}\tilde u_1\tilde u_2\tilde U_3d\Gamma\right\|_{L^2_{k_0}L^2_{\tau_0}}.
\end{equation}
Similarly to $X^{-1/2,-1/2}$ case, we  have $\displaystyle \sup_{i=0,1,2,3}L_i (=:L_s) \gtrsim {n_{1}}^2{n_{4}}^{-1}$, and so
\begin{equation}\label{eq:est of freq}
\frac{\left|k_1k_2\right|^{1/2}}{\left|k_0k_3\right|^{1/2}L^{1/2}_s} \lesssim \frac{{n_{1}}}{{k_{0}}^{1/2}{k_{3}}^{1/2}}\cdot \frac{{n_{4}}^{1/2}}{{n_{1}}} \le \frac{1}{{k_{3}}^{1/2}}.
\end{equation}
We first consider $\displaystyle  \sup_{i=1,2,3}L_i= L_1$. By combining (\ref{eq:i case after holder}) and (\ref{eq:est of freq}), it suffices to show that
\begin{equation}\label{eq:i case claim}
\left\|\sum_{|\bf k|  \in \mathcal{NR} } \int \frac{1}{{k_{3}}^{1/2}L_0^{1/2-2\delta}L^{1/2}_2L^{1/2}_3}\tilde u_1\tilde u_2\tilde U_3d\Gamma\right\|_{L^2_{k_0}L^2_{\tau_0}} \lesssim \left\|u_1\right\|_{L^2_{x,t}}\left\|u_2\right\|_{L^2_{x,t}}\left\|U_3\right\|_{L^2_{x,t}}. 
\end{equation}
Then by duality, it suffices to prove
\begin{equation}\label{the desire}
\left|\iint u_0u_1u_2 U_3dxdt\right| \lesssim \left\|u_0\right\|_{X^{0,1/2-2\delta}}\left\|u_1\right\|_{L^2_{x,t}}\left\|u_2\right\|_{X^{0,1/2}}\left\|U_3\right\|_{X^{1/2,1/2}}.
\end{equation}
We can obtain \eqref{the desire} by the Cauchy-Schwarz inequality and (\ref{eq:L2 multiest.}) for small enough $\delta > 0$. By symmetry, $\displaystyle \sup_{i=0,1,2,3}L_i = L_2$ or $L_3$ cases are proved as well.\\
\\
Finally, we assume that $\displaystyle \sup_{i=0,1,2,3}L_i  = L_0 \gtrsim {n_{1}}^2{n_{4}}^{-1} $. From \eqref{eq:est of freq}, the left hand side of (\ref{eq:claim of L2L1}) is bounded by 
\begin{equation}\label{eq:ii claim}
\left\|\sum_{|\bf k|  \in \mathcal{NR} } \int \frac{1}{{n_{3}}^{1/2}L_0^{1/2-\delta}L^{1/2}_1L^{1/2}_2L^{1/2}_3}\tilde u_1 \tilde u_2 \tilde U_3d\Gamma\right\|_{L^2_{k_0}L^1_{\tau_0}}.
\end{equation}
From the assumption $L_0 \gtrsim L_1$, the Fubini theorem and the Cauchy-Schwarz inequality,
\begin{equation}\label{eq:step of ii claim}
\begin{aligned}
\eqref{eq:ii claim} &\le \left\|\sum_{|\bf k|  \in \mathcal{NR} } \int \frac{1}{{n_{3}}^{1/2}  L_1^{1-\delta} L_2^{1/2}L_3^{1/2}}\tilde u_1 \tilde u_2 \tilde U_3d\Gamma\right\|_{L^2_{k_0}L^1_{\tau_1}} \\
&\lesssim \left\|\sum_{|\bf k|  \in \mathcal{NR} } \int \frac{1}{{n_{3}}^{1/2}  L_1^{1/2-2\delta} L_2^{1/2}L_3^{1/2}}\tilde u_1 \tilde u_2 \tilde U_3d\Gamma\right\|_{L^2_{k_0}L^2_{\tau_1}} .
\end{aligned}
\end{equation}
The last term of \eqref{eq:step of ii claim} is similar to the left hand side of (\ref{eq:i case claim}), so we finish the proof.
\end{proof}
\begin{proof}[Proof of Lemma~\ref{lem:tri est.}]
It is obtained from the combination of Lemma \ref{lem:tri resonant est} and \ref{lem:local tri. est.} with summation with respect to each of dyadic frequency supports. More precisely, we observe  
\begin{equation*}
\left(uv- P_0\left(uv\right)\right)\partial_x^{-1}w = F_{r}\left(u,v,w\right) + F_{r}\left(v,u,w\right) + F_{nr}\left(u,v,w\right),
\end{equation*}
and use the fact that left hand side of (\ref{eq:freq supp nonresonant est.}) vanishes unless $n_1 \sim n_2$ when we sum up with dyadic numbers. 
\end{proof}
\begin{rmk}\label{rmk:rescaling est.}
In the proof of main theorems, we will use a rescaling argument. The bilinear and trilinear estimates obtained above can be easily restated with a rescaling parameter. We record facts here for convenience of readers. See more details in \cite{Bourgain:1993cl}, \cite{Colliander:2003td}, \cite{Colliander:2004gc} and \cite{university2007well}.
We let $\alpha \MB{T} = [0, 2 \pi \alpha)$ be the spatial domain. Then implicit constants of \eqref{eq:X^s,1/2 est.} and \eqref{eq:L2 multiest.} depend on $\alpha$. More precisely, for $2\pi \alpha $-periodic function $ u$, we define
\begin{equation}\label{eq:den of rescaled H^s norm }
\left\|u\right\|_{H^s(\alpha \MB{T})} := \frac{1}{\left(2\pi \alpha\right)^{1/2}} \left(\sum_{k \in \MB{Z}/{\alpha}} \left<k\right>^{2s}\left|\hat u\right|^2\right)^{1/2}, \hat u(k) = \int_0^{2\pi \alpha} e^{-ikx} u(x) \,dx 
\end{equation} 
and 
\begin{equation*}
\left\|u\right\|_{X^{s,b}\left(\alpha \MB{T}\right)}:=\left\|\left<k\right>^s \left<\tau-k^3\right>^b \MC{F} \left(u\right) \left(k, \tau \right)\right\|_{L^2_k\left(\MB{Z}/\alpha\right) L^2_{\tau}}.
\end{equation*}
In addition, we can define $Y^s\left(\alpha \MB{T}\right)$ and $Z^s\left(\alpha \MB{T}\right)$ norm by the same method. 
From \cite{Colliander:2004gc}, we have
\begin{equation}\label{eq: rescaled X^s,b est.}
\left\|uv\right\|_{X^{-1/2,1/2}\left(\alpha \MB{T}\right)} \lesssim \alpha^{0+} \left\|u\right\|_{Y^{1/2}\left(\alpha \MB{T}\right)} \left\|v\right\|_{Y^{1/2}\left(\alpha \MB{T}\right)}.
\end{equation}
Moreover, the following estimates are well-known,
\begin{equation}\label{eq:stric est. of recaled}
\left\|u\right\|_{L^4_{x,t}\left(\alpha \MB{T}\right)} \lesssim C\left(\alpha\right)\left\|u\right\|_{X^{0,1/3}\left(\alpha \MB{T}\right)}
\end{equation}
and
\begin{equation}\label{eq:another stric est. of rescaled}
\left\|u\right\|_{L^{\infty}_{x,t}\left(\alpha \MB{T}\right)} \lesssim C\left(\alpha\right)\left\|u\right\|_{X^{\frac{1}{2}+,\frac{1}{2}+}\left(\alpha \MB{T}\right)},
\end{equation}
where implicit constants $C\left(\alpha\right)$ are decreasing functions of $\alpha$. In particular, we have $C\left(\alpha\right) \le C\left(1\right)$ for $\alpha \ge 1$.
From \eqref{eq:stric est. of recaled} and \eqref{eq:another stric est. of rescaled}, 
\begin{equation}\label{eq:tri est. of direct method}
\left\|uvw\right\|_{L^2_{x,t}\left(\alpha \MB{T}\right)} \lesssim C'\left(\alpha\right) \left\|u\right\|_{X^{0,1/3}\left(\alpha \MB{T}\right)}\left\|v\right\|_{X^{0,1/3}\left(\alpha \MB{T}\right)}\left\|w\right\|_{X^{\frac{1}{2}+,\frac{1}{2}+}\left(\alpha \MB{T}\right)},
\end{equation}
where implicit constant $C'\left(\alpha\right)$ is also a decreasing function of $\alpha$.
Moreover, by rescaling \eqref{eq:L2 multiest.}, we can obtain
\begin{equation}\label{eq:tri est. of rescaled}
\left\|uvw\right\|_{L^2_{x,t}\left(\alpha \MB{T}\right)} \lesssim \alpha^M \left\|u\right\|_{X^{0,1/2-\delta}\left(\alpha \MB{T}\right)}\left\|v\right\|_{X^{0,1/2-\delta}\left(\alpha \MB{T}\right)}\left\|w\right\|_{X^{1/2-\delta,1/2-\delta}\left(\alpha \MB{T}\right)},
\end{equation}
for some positive constant $M$. Interpolating \eqref{eq:tri est. of direct method} and \eqref{eq:tri est. of rescaled} we can obtain the $\alpha$-rescaled estimate as follows
\begin{equation}\label{eq:rescaled L2 multi est}
\left\|uvw\right\|_{L^2_{x,t}\left(\alpha \MB{T}\right)} \lesssim \alpha^{0+} \left\|u\right\|_{X^{0,1/2-\sigma}\left(\alpha \MB{T}\right)}\left\|v\right\|_{X^{0,1/2-\sigma}\left(\alpha \MB{T}\right)}\left\|w\right\|_{X^{1/2-\sigma,1/2-\sigma}},
\end{equation}
for some small $0<\sigma \ll1$. Once we obtain \eqref{eq:den of rescaled H^s norm }-\eqref{eq:stric est. of recaled} and \eqref{eq:rescaled L2 multi est}, it is straightforward that one can replace \eqref{eq:X^s,1/2 est.} and \eqref{eq:L2 multiest.} in the proof, and so conclude \eqref{eq: bi est.} and \eqref{eq:tri est.} with the scaling parameter $\alpha$. 
\end{rmk}

\section{Differentiation by parts and Proof of Theorem \ref{thm:est. of different data} and \ref{thm:est. of truncated flow}}\label{Section: proof of different flow}
In this section, we use the method of the normal form to show Theorem \ref{thm:est. of different data} and \ref{thm:est. of truncated flow}. The normal form is performed via the differentiation by parts. Writing the system \eqref{eq:Coupled KdV system} in the interaction representation, we take the differentiation by parts to change the quadratic nonlinear terms into the bilinear nonlinear terms as the boundary term and the trilinear terms. This procedure replaces the use of the Miura transform in the proof of \cite{Colliander:2005vv}. 
\subsection{Differentiation by parts}
To simplify the notation, we denote
\begin{align*}
u\left(t,x\right)=S_{CKdV}\left(t\right) u_0, \quad v\left(t,x\right)=S_{CKdV}\left(t\right) v_0, \\
u^b\left(t,x\right)=S_{BKdV}\left(t\right) u_0, \quad v^b\left(t,x\right)=S_{BKdV}\left(t\right) v_0.
\end{align*}
Moreover, denote ${\bf u} = e^{t \partial_x^3}u$, ${\bf v}  = e^{t \partial_x^3}v$, ${\bf u}^b = e^{t \partial_x^3}u^b$ and ${\bf v}^b  = e^{t \partial_x^3}v^b$.
From (\ref{eq:Coupled KdV system}),
\begin{align*}
\partial_t{\bf u} = e^{t \partial_x^3}\left(\partial_x^3 u+\partial_t u\right) = - \frac{e^{t \partial_x^3}}{2}\partial_x\left(vv\right) = -\frac{e^{t\partial_x^3}}{2}  \partial_x\left(e^{-t\partial_x^3 }{\bf v} \cdot e^{-t\partial_x^3 }{\bf v}\right),  \\
\partial_t{\bf v} = e^{t \partial_x^3}\left(\partial_x^3 v+\partial_t v\right) = - e^{t \partial_x^3}\partial_x\left(uv\right) = -e^{t\partial_x^3} \partial_x\left(e^{-t\partial_x^3 }{\bf u} \cdot e^{-t\partial_x^3 }{\bf v}\right).
\end{align*}
We look at the system of the Fourier variables, still denoted as ${\bf u}_k$, ${\bf v}_k$ for $ k \in \mathbb Z^*$,\footnote{$k_i \not =0$ in the sequel is due to mean zero assumption.}
\begin{equation*}\label{eq: interaction repre}
\begin{aligned}
\partial_t {\bf u}_k = -\frac{i}{2} e^{-ik^3t} \sum_{k_1+k_2=k} k e^{ik_1^3 t} {\bf v}_{k_1} e^{ik_2^3 t} {\bf v}_{k_2} = -\frac{i}{2}\sum_{k_1+k_2=k} k e^{-i \phi\left({\rm k}\right) t} {\bf v}_{k_1} {\bf v}_{k_2}, \\
\partial_t{\bf v}_k = -i e^{-ik^3t} \sum_{k_1+k_2=k} k e^{ik_1^3 t} {\bf u}_{k_1} e^{ik_2^3 t} {\bf v}_{k_2} = -i\sum_{k_1+k_2=k} k e^{-i \phi\left({\rm k}\right) t} {\bf u}_{k_1} {\bf v}_{k_2},
\end{aligned}
\end{equation*}
where $\phi\left({\rm k}\right)= \phi\left(k_1,k_2\right)=3k_1k_2\left(k_1+k_2\right)$. Taking the differentiation by parts, we write
\begin{equation}\label{eq:diff by parts original flow}
\begin{aligned}
\partial_t{\bf u}_k = &-\frac{i}{2} \left[\partial_t\left\{\sum_{k_1+k_2=k}k \frac{e^{-i \phi\left({\rm k}\right)t}}{-i \phi\left({\rm k}\right)} {\bf v}_{k_1} {\bf v}_{k_2}\right\} \right. \\ 
&\left. +2 \sum_{k_1+k_3=k}k \frac{e^{-i \phi\left(k_1,k_3\right)t}}{-i \phi\left(k_1,k_3\right)} {\bf v}_{k_3}\sum_{k_{11}+k_{12}=k_1}ik_1 e^{-i \phi\left(k_{11},k_{12}\right)t} {\bf u}_{k_{11}}{\bf v}_{k_{12}}\right] \\
=& \frac{1}{6}\left[\partial_t\left\{\sum_{k_1+k_2=k} \frac{e^{-i\phi\left({\rm k}\right)t}}{k_1k_2}{\bf v}_{k_1}{\bf v}_{k_2}\right\} +2i \sum_{\substack{k_1+k_2+k_3=k \\ k_1+k_2 \not = 0}} \frac{e^{-i\Phi\left({\rm k}\right)t}}{k_3}{\bf u}_{k_1}{\bf v}_{k_2}{\bf v}_{k_3}\right] \\
\partial_t{\bf v}_k = &- i\left[\partial_t\left\{\sum_{k_1+k_2=k}k \frac{e^{-i \phi\left({\rm k}\right)t}}{-i \phi\left({\rm k}\right)} {\bf u}_{k_1} {\bf v}_{k_2}\right\} \right. \\ 
&\left. + \sum_{k_1+k_3=k}k \frac{e^{-i \phi\left(k_1,k_3\right)t}}{-i \phi\left(k_1,k_3\right)} {\bf u}_{k_1}\sum_{k_{21}+k_{22}=k_3}ik_3 e^{-i \phi\left(k_{21},k_{22}\right)t} {\bf u}_{k_{21}}{\bf v}_{k_{22}}\right. \\
&\left. + \frac{1}{2}\sum_{k_1+k_3=k}k \frac{e^{-i \phi\left(k_1,k_3\right)t}}{-i \phi\left(k_1,k_3\right)} {\bf v}_{k_3}\sum_{k_{11}+k_{12}=k_1}ik_1 e^{-i \phi\left(k_{11},k_{12}\right)t} {\bf v}_{k_{11}}{\bf v}_{k_{12}}\right] \\
=& \frac{1}{3}\left[\partial_t\left\{\sum_{k_1+k_2=k} \frac{e^{-i\phi\left({\rm k}\right)t}}{k_1k_2}{\bf u}_{k_1}{\bf v}_{k_2}\right\} +\sum_{\substack{k_1+k_2+k_3=k \\ k_1+k_2 \not = 0}} \frac{ie^{-i \Phi\left({\rm k}\right)t}}{k_3}\left({\bf v}_{k_1}{\bf u}_{k_2}{\bf u}_{k_3}+\frac{1}{2}{\bf v}_{k_1}{\bf v}_{k_2}{\bf v}_{k_3}\right)\right], 
\end{aligned}
\end{equation}
where $\Phi\left({\rm k}\right)=\Phi\left(k_1,k_2,k_3\right)=3\left(k_1+k_2\right)\left(k_2+k_3\right)\left(k_3+k_1\right)$. Similarly, we write a system for (\ref{eq:smooth truncated system}),
\begin{equation}\label{eq:diff by parts truncated flow}
\begin{aligned}
\partial_t{\bf u}^b_k =& \frac{1}{6}\left[\partial_t\left\{\sum_{k_1+k_2=k} \frac{e^{-i\phi\left({\rm k}\right)t}}{k_1k_2}b\left(k\right){\bf v}^b_{k_1}{\bf v}^b_{k_2}\right\} +2i \sum_{\substack{k_1+k_2+k_3=k \\ k_1+k_2 \not = 0}} \frac{e^{-i\Phi\left({\rm k}\right)t}}{k_3}b\left(k\right)b\left(k_1+k_2\right){\bf u}^b_{k_1}{\bf v}^b_{k_2}{\bf v}^b_{k_3}\right] \\
\partial_t{\bf v}^b_k =& \frac{1}{3}\left[\partial_t\left\{\sum_{k_1+k_2=k} \frac{e^{-i\phi\left({\rm k}\right)t}}{k_1k_2}b\left(k\right){\bf v}^b_{k_1}{\bf v}^b_{k_2}\right\} \right.\\ 
&\left.+\sum_{\substack{k_1+k_2+k_3=k \\ k_1+k_2 \not = 0}} \frac{ie^{-i\Phi\left({\rm k}\right)t}}{k_3}b\left(k\right)b\left(k_1+k_2\right)\left({\bf v}^b_{k_1}{\bf u}^b_{k_2}{\bf u}^b_{k_3}+\frac{1}{2}{\bf v}^b_{k_1}{\bf v}^b_{k_2}{\bf v}^b_{k_3}\right)\right].
\end{aligned}
\end{equation}
Integrating (\ref{eq:diff by parts original flow}) and (\ref{eq:diff by parts truncated flow}) in time $t$, we have
\begin{equation}\label{eq:solution u}
\begin{split}
{\bf u}_k\left(t\right)=&{\bf u}_k\left(0\right)+\frac{1}{6}\left[\sum_{k_1+k_2=k}\frac{e^{-i\phi\left({\rm k}\right)t}}{k_1k_2}{\bf v}_{k_1}\left(t\right){\bf v}_{k_2}\left(t\right)-\sum_{k_1+k_2=k}\frac{1}{k_1k_2}{\bf v}_{k_1}\left(0\right){\bf v}_{k_2}\left(0\right) \right. \\
&\left.+ 2i \int_0^t \sum_{\substack{k_1+k_2+k_3=k \\ k_1+k_2 \not = 0}}\frac{e^{-i\Phi\left({\rm k}\right)s}}{k_3}{\bf u}_{k_1}{\bf v}_{k_2}{\bf v}_{k_3}ds\right] 
\end{split}
\end{equation}
\begin{equation}\label{eq:solution v}
\begin{split}
{\bf v}_k\left(t\right)=&{\bf v}_k\left(0\right)+\frac{1}{3}\left[\sum_{k_1+k_2=k}\frac{e^{-i\phi\left({\rm k}\right)t}}{k_1k_2}{\bf u}_{k_1}\left(t\right){\bf v}_{k_2}\left(t\right)-\sum_{k_1+k_2=k}\frac{1}{k_1k_2}{\bf u}_{k_1}\left(0\right){\bf v}_{k_2}\left(0\right) \right.\\
&\left.+ \int_0^t \sum_{\substack{k_1+k_2+k_3=k \\ k_1+k_2 \not = 0}}\frac{i e^{-i\Phi\left({\rm k}\right)s}}{k_3}\left({\bf v}_{k_1}{\bf u}_{k_2}{\bf u}_{k_3}+\frac{1}{2}{\bf v}_{k_1}{\bf v}_{k_2}{\bf v}_{k_3}\right)ds\right]
\end{split}
\end{equation}
\begin{equation}\label{eq:solution ub}
\begin{split}
{\bf u}_k^b\left(t\right)=&{\bf u}_k\left(0\right)+\frac{1}{6}\left[\sum_{k_1+k_2=k}\frac{e^{-i\phi\left({\rm k}\right)t}}{k_1k_2}b\left(k\right){\bf v}^b_{k_1}\left(t\right){\bf v}^b_{k_2}\left(t\right)-\sum_{k_1+k_2=k}\frac{1}{k_1k_2}b\left(k\right){\bf v}_{k_1}\left(0\right){\bf v}_{k_2}\left(0\right) \right. \\
&\left.+ 2i\int_0^t \sum_{\substack{k_1+k_2+k_3=k \\ k_1+k_2 \not = 0}}\frac{e^{-i\Phi\left({\rm k}\right)s}}{k_3}b\left(k\right)b\left(k_1+k_2\right){\bf u}^b_{k_1}{\bf v}^b_{k_2}{\bf v}^b_{k_3}ds\right] 
\end{split}
\end{equation}
\begin{equation}\label{eq:solution vb}
\begin{split}
{\bf v}_k^b\left(t\right)=&{\bf v}_k\left(0\right)+\frac{1}{3}\left[\sum_{k_1+k_2=k}\frac{e^{-i\phi\left({\rm k}\right)t}}{k_1k_2}b\left(k\right){\bf u}^b_{k_1}\left(t\right){\bf v}^b_{k_2}\left(t\right)-\sum_{k_1+k_2=k}\frac{1}{k_1k_2}b\left(k\right){\bf u}_{k_1}\left(0\right){\bf v}_{k_2}\left(0\right) \right.\\
&\left.+ \int_0^s \sum_{\substack{k_1+k_2+k_3=k \\ k_1+k_2 \not = 0}}\frac{ie^{-i\Phi\left({\rm k}\right)s}}{k_3}b\left(k\right)b\left(k_1+k_2\right)\left({\bf v}^b_{k_1}{\bf u}^b_{k_2}{\bf u}^b_{k_3}+\frac{1}{2}{\bf v}^b_{k_1}{\bf v}^b_{k_2}{\bf v}^b_{k_3}\right)ds\right].
\end{split}
\end{equation}
Transforming back (\ref{eq:solution u}), we write $u_k\left(t\right)$ as follows,
\begin{equation}\label{eq:trans back soln}
\begin{split}
u_k\left(t\right)=e^{ik^3t}u_{k}\left(0\right)+\frac{e^{ik^3t}}{6}\left[\sum_{k_1+k_2=k}\frac{e^{-i\phi\left({\rm k}\right)t}}{k_1k_2}e^{-it\left(k_1^3+k_2^3\right)}v_{k_1}\left(t\right)v_{k_2}\left(t\right)  -\sum_{k_1+k_2=k}\frac{v_{k_1}\left(0\right)v_{k_2}\left(0\right)}{k_1k_2} \right. \\
\left.+ 2i\int_0^t \sum_{\substack{k_1+k_2+k_3=k \\ k_1+k_2 \not = 0}}\frac{e^{-i\Phi\left({\rm k}\right)s}}{k_3}e^{-is\left(k_1^3+k_2^3+k_3^3\right)}u_{k_1}v_{k_2}v_{k_3}ds\right].
\end{split}
\end{equation}
In the same way, the solutions $v_k\left(t\right)$, $u^b_k\left(t\right)$ and $v^b_k\left(t\right)$ are denoted similarly. As shown in the theorems, we should investigate in detail the difference between solutions. Since $v$ and $v^b$ are handled similarly, we mainly consider the solutions $u$ and $u^b$. Note that `$k_1+k_2 \not = 0$' in (\ref{eq:trans back soln}) means `$-P_0\left(uv\right)\partial_x^{-1}v$' in the spatial domain. This represent why we require the trilinear form in Lemma \ref{lem:tri est.} in the approximation analysis.

\subsection{Proof of Theorem \ref{thm:est. of truncated flow}} \label{sec: pf of thm trun} $ $ \\
We prove the estimate \eqref{eq:est. of truncated flow} for $ u-u_b$. By the following argument, we can obtain \eqref{eq:est. of truncated flow} for $v-v_b$ as well. From \eqref{eq:space embedding} and \eqref{eq:solution u}-\eqref{eq:trans back soln}, the first term of the left hand side of (\ref{eq:est. of truncated flow}) is bounded by
\begin{equation}\label{eq:Full est.}
\begin{split}
&\left\|P_{\le N^{1/2}}\left(u-u^b\right)\right\|_{Y^{-1/2}}\\
& \lesssim \left\|P_{ \le N^{1/2}}\left[ \partial_x^{-1}v\partial_x^{-1}v-\partial_x^{-1}v^b\partial_x^{-1}v^b\right]\right\|_{Y^{-1/2}} \\
&+ \left\|P_{\le N^{1/2}} \left[\left(uv-P_0\left(uv\right)\right)\partial^{-1}_xv-\left(B\left(u^bv^b\right)-P_0\left(B\left(u^bv^b\right)\right)\right)\partial^{-1}_xv^b\right]\right\|_{Z^{-1/2}} \\
&=: \left\|B_2\left(v,v\right)\right\|_{Y^{-1/2}} +\left\|N_3\left(u,v,v\right)\right\|_{Z^{-1/2}}.
\end{split}
\end{equation}
We now use the bilinear and trilinear estimates obtained in Section 3. 
From the triangle inequality and (\ref{eq: bi est.}),
\begin{equation}\label{eq:B_2 est.}
\begin{split}
\left\|B_2\left(v,v\right)\right\|_{Y^{-1/2}}
&\le \left\|P_{ \le N^{1/2}}\left[\partial_x^{-1}v\partial_x^{-1}\left(v - v^b\right)\right]\right\|_{Y^{-1/2}} + \left\|P_{ \le N^{1/2}}\left[\partial_x^{-1}v^b\partial_x^{-1}\left(v - v^b\right)\right]\right\|_{Y^{-1/2}} \\
&\lesssim \left\|P_{ \le N^{1/2}}v\right\|_{Y^{-1/2}} \left\|P_{ \le N^{1/2}}\left(v - v^b\right)\right\|_{Y^{-1/2}} \\
&+ \left\|P_{ \le N^{1/2}}v^b\right\|_{Y^{-1/2}} \left\|P_{ \le N^{1/2}}\left(v - v^b \right)\right\|_{Y^{-1/2}} \\
&+ \text{(remainder terms)$_1$}.
\end{split}
\end{equation}
Here $\text{(remainder terms)$_1$}$ contain \emph{high-high} to \emph{low} frequency interactions. \\
\noindent The integral terms are also estimated by the triangle inequality, (\ref{eq:tri est.}) and $P_0B = P_0$,
\begin{equation}\label{eq:N_3 est.}
\begin{split}
\left\|N_3\left(u,v,v\right)\right\|_{Z^{-1/2}}
&\lesssim \left\|P_{ \le N^{1/2}}u\right\|_{Y^{-1/2}}\left\|P_{ \le N^{1/2}}v\right\|_{Y^{-1/2}} \left\|P_{ \le N^{1/2}}\left(v - v^b\right)\right\|_{Y^{-1/2}} \\
&+  \left\|P_{ \le N^{1/2}}u\right\|_{Y^{-1/2}}\left\|P_{ \le N^{1/2}}v^b\right\|_{Y^{-1/2}} \left\|P_{ \le N^{1/2}}\left(v - v^b\right)\right\|_{Y^{-1/2}} \\
&+ \left\|P_{ \le N^{1/2}}v^b\right\|_{Y^{-1/2}}\left\|P_{ \le N^{1/2}}v^b\right\|_{Y^{-1/2}} \left\|P_{ \le N^{1/2}}\left(u - u^b\right)\right\|_{Y^{-1/2}} \\
&+ \left\|P_{\le N^{1/2}}\MC{F}_x^{-1}\left[\sum_{\substack{k_1+k_2+k_3=k \\ k_1+k_2 \not = 0}}\frac{\left(1-b\left(k_1+k_2\right)\right)}{k_3}u^b_{k_1}v^b_{k_2}v^b_{k_3}\right]\right\|_{Z^{-1/2}} \\
&+ \text{(remainder terms)$_2$}.
\end{split}
\end{equation}
We take three steps. We first show that the remainder terms are $ { \scriptstyle\MC{O}}_N\left(1\right)$ using the bilinear and trilinear estimates that is obtained the last section. Next, we show that $Z^{-1/2}$-term of (\ref{eq:N_3 est.}) is ${ \scriptstyle\MC{O}}_N\left(1\right)$ as well. Lastly, we show that terms involving the difference are absorbed into the left hand side of \eqref{eq:Full est.} and $\left\|P_{\le N^{1/2}}\left(v-v^b\right)\right\|_{Y^{-1/2}}$. For this step, we use a rescaling argument to make the factor $\|P_{\le N^{1/2}} u \|_{Y^{-1/2}} $ small in a large domain.  \\
\\
{\bf Step 1}. \\ 
First, we handle \text{(remainder terms)$_1$} in the boundary terms. We take a dyadic decomposition and use Lemma \ref{lem:bi est.}, Lemma \ref{lem:tri resonant est} and Lemma \ref{lem:local tri. est.}. In view of \eqref{eq:Full est.}, \text{(remainder terms)$_1$} contains only \emph{high-high} to \emph{low} interactions. Namely, it is bounded by
\begin{equation*}
\left\|P_{\le N^{1/2}}F_2\left(v_{hi},v_{hi}\right)\right\|_{Y^{-1/2}}, 
\end{equation*}
where we denote $v_{low} = P_{\le N^{1/2}}v$ and $v_{hi} = \left(1-P_{\le N^{1/2}}\right)v$. Obviously, $v$ can be replaced to $u$, $u^b$ or $v^b$, but they are handled in the same way. From Lemma \ref{lem:bi est.} and the global well-posedness, since
\begin{equation*}
\left\|F_2\left(v,v\right)\right\|_{Y^{-1/2}} \lesssim \|v\|_{Y^{-1/2}} \|v\|_{Y^{-1/2}},
\end{equation*}
we have 
\begin{equation}\label{eq: result of bi est.}
\left\|P_{\le N^{1/2}}F_2\left(v_{hi},v_{hi}\right)\right\|_{Y^{-1/2}} \lesssim \|v_{hi}\|_{Y^{-1/2}} \|v_{hi}\|_{Y^{-1/2}}  \sim { \scriptstyle\MC{O}}_N\left(1\right).
\end{equation}
\\
Next, we control the integral terms. Similarly, (remainder terms)$_2$ also has the sum of the multilinear terms, but it has resonant form $P_{\le N^{1/2}} F_r \left(u,v,v\right)$ and nonresonant form $P_{\le N^{1/2}} F_{nr} \left(u,v,v\right)$. The resonant case can be controlled as the boundary terms did. More precisely, we write the $P_{\le N^{1/2}} F_{r}\left(u,v,v\right)$ as follows,
\begin{equation}\label{eq:truncated resonant}
\sum_{N_0,N_1,N_2,N_3} P_{N_0} P_{\le N^{1/2}} F_{r}\left(P_{N_1}u,P_{N_2}v,P_{N_3}v\right).
\end{equation}
Likewise, $u$ and $v$ in the sequel can be replaced by $u^b$ or $v^b$. As before, in view of \eqref{eq:N_3 est.}, (remainder terms)$_2$ does not contain trilinear terms of which all factors are from low frequency piece.
We thus have $n_1 > N^{1/2}$ and from \eqref{eq:resonant trilinear}, we can write the form \eqref{eq:truncated resonant} in (remainder terms)$_2$ as follows,
\begin{equation}\label{eq: low-hi-hi tri}
\sum_{N_0,N_2} P_{N_0} P_{\le N^{1/2 }}F_{r}\left(P_{N_0}u_{low}, P_{N_2}v_{hi}, P_{N_2} v_{hi}\right) = P_{\le N^{1/2}}F_r\left(u_{low}, v_{hi}, v_{hi}\right).
\end{equation}

Similarly to the boundary case, we have the following estimates by Lemma~\ref{lem:tri resonant est} and the global well-posedness,
\begin{equation*}
\left\|F_{r}\left(u_{low}, v, v\right)\right\|_{Z^{-1/2}} \lesssim 1,
\end{equation*}
and
\begin{equation}\label{eq: high freq resonant est.}
\left\|P_{\le N^{1/2}}F_r\left(u_{low}, v_{hi}, v_{hi}\right)\right\|_{Z^{-1/2}}\lesssim \|u\|_{Y^{-1/2}}\|v_{hi}\|_{Y^{-1/2}} \|v_{hi}\|_{Y^{-1/2}} \sim { \scriptstyle\MC{O}}_N\left(1\right).
\end{equation}
In other words, (remainder terms)$_2$ is bounded by ${ \scriptstyle\MC{O}}_N\left(1\right)$.\\
\\
The integral terms associated with nonresonant case require a bit more work since (\ref{eq:freq supp nonresonant est.}) has $N_i$ and $n_3$ as its coefficients. For the frequency interval $\left[N^{1/2},2N^{1/2}\right]$, we can divide this interval into $O(\left(N'\right)^{1/4})$ intervals uniformly, and then by the  orthogonality and the pigeon-hole principle, there exists at least one interval form of $[M,M+N^{1/4}]$ such that
\begin{equation*}
\left\|\left(P_{\le M+N^{1/4}}-P_{\le M}\right)u\right\|_{Y^{-1/2}} \lesssim N^{-\sigma}.
\end{equation*}
Fix this $M$, we can let $u_{low} = P_{\le M} u$, $u_{med} = \left(P_{\le M+N^{1/4}}-P_{\le M}\right)u$ and $u_{hi} = \left(1-P_{\le M+N^{1/4}}\right)u$. Then by Lemma \ref{lem:tri est.}, (remainder terms)$_2$ is bounded by $\MC{O}\left(N^{-\sigma}\right)$ if it has $u_{med}$ terms. Moreover, as before, the terms consisting of low frequency terms only are not included in (remainder terms)$_2$. Hence, we consider terms which have at least one $u_{hi}$. The worst case of this situation is $P_{\le M}F_{nr}\left(u_{hi}v_{low}v_{low}\right)$. As like the boundary terms, we split the solutions into the dyadic pieces,
\begin{equation*}
\sum_{N_0,N_1,N_2,N_3} P_{N_0} P_{\le M} F_{nr}\left(P_{N_1}u_{hi},P_{N_2}v_{low},P_{N_3}v_{low}\right).
\end{equation*}
Using frequency relation, we have $\left|k_1+k_2+k_3\right| =\left|k_0\right| \le M$, $\left|k_1\right| \ge M+N^{1/4}$ and $\left|k_2\right|, \left|k_3\right| \le M$. We thus have $n_3 \gtrsim N^{1/4}$, and from Lemma \ref{lem:local tri. est.}, 
\begin{equation}\label{eq: result of trun. trilinear}
\left\|\sum_{N_0,N_1,N_2,N_3} P_{N_0} P_{\le M} F_{nr}\left(P_{N_1}u_{hi},P_{N_2}v_{low},P_{N_3}v_{low}\right)\right\|_{Z^{-1/2}} \lesssim N^{-\sigma}.
\end{equation}
Therefore, (remainder terms)$_2$ is bounded by ${ \scriptstyle\MC{O}}_N\left(1\right)$.\\
\\
{\bf Step 2}. \\ 
The argument is based on the mean value theorem. Indeed, we use the smooth truncation instead of the sharp truncation to applying the mean value theorem. Our claim is
\begin{equation}\label{eq:goal of MVT}
\begin{aligned}
\left\|P_{\le N^{1/2}}\MC{F}^{-1}\left[\sum_{\substack{k_1+k_2+k_3=k \\ k_1+k_2 \not = 0}}\frac{\left(1-b\left(k_1+k_2\right)\right)}{k_3}v^b_{k_1}u^b_{k_2}u^b_{k_3}\right]\right\|_{Z^{-1/2}} \\
\le C\left(T,\left\|u^b_0\right\|_{H^{-1/2}_0},\left\|v^b_0\right\|_{H^{-1/2}_0}\right) { \scriptstyle\MC{O}}_N\left(1\right).
\end{aligned}
\end{equation}
To prove \eqref{eq:goal of MVT}, we inspect the support of indices. By the sharp truncation $P_{\le N^{1/2}}$ and the smooth truncation $b\left(k_1+k_2\right)$, we have $\left|k_1+k_2+k_3\right| =\left|k\right| \lesssim N^{1/2}$, $\left|k_1+k_2\right| \gtrsim N$ and then $\left|k_3\right| \gtrsim N$, and therefore, $n_1 \gtrsim N$. If $\left(k_1,k_2,k_3\right)$ is nonresonant, then  (\ref{eq:goal of MVT}) is directly obtained by Lemma \ref{lem:local tri. est.} and the global well-posedness. The remaining case is the resonant case. As mentioned above, we have $\left|k_1+k_2\right| \gtrsim N$ and $\left|k_3\right| \gtrsim N$, so there are only two cases. That is, $\left(k_1,k_2,k_3\right) = \left(k,-k_3,k_3\right) \text{or} \left(-k_3, k, k_3\right)$. For fixed time $T$, we have
\begin{equation*}
\left|\sum_{k \in \MB{Z}^*} \frac{1}{k} v_{-k}u_{k}\right|   \lesssim  \left\|v\right\|_{H^{-1/2}_0}\left\|u\right\|_{H^{-1/2}_0} \lesssim 1
\end{equation*}
by the Cauchy-Schwarz inequality and the global well-posedness, and then 
\begin{equation*}
\frac{1}{2} \left|\sum_{k \in \MB{Z}^*} \frac{v_{-k}u_{k} - v_{k}u_{-k}}{k} \right| =  \left|\sum_{k \in \MB{Z}^*} \frac{1}{k} v_{-k}u_{k}\right| \lesssim 1.
\end{equation*}
Hence, we can let $e_{k}$ such that 
\begin{equation}\label{eq:error sum}
e_{k} = \frac{1}{k} v_{-k}u_{k} - \frac{1}{k} v_{k}u_{-k} ~\text{and}~ \sum_{k \gtrsim N} e_{k} = { \scriptstyle\MC{O}}_N\left(1\right).
\end{equation}
We can now rewrite a piece of the resonant case as follows,
\begin{equation*}
\begin{aligned}
&\sum_{\left|k_3\right| \gtrsim N}1_{\left[-N^{1/2},N^{1/2}\right]}\left(k\right)\frac{1-b\left(k-k_3\right)}{k_3}v^b_{k}u^b_{-k_3}u^b_{k_3} \\ 
&+\sum_{\left|k_3\right| \gtrsim N}1_{\left[-N^{1/2},N^{1/2}\right]}\left(k\right)\frac{1-b\left(k-k_3\right)}{k_3}v^b_{-k_3}u^b_{k}u^b_{k_3}  \\
&=\sum_{k_3 \gtrsim N}1_{\left[-N^{1/2},N^{1/2}\right]}\left(k\right)\frac{b\left(k+k_3\right)-b\left(k-k_3\right)}{k_3}v^b_{k}u^b_{-k_3}u^b_{k_3} \\ 
&+\sum_{k_3 \gtrsim N}1_{\left[-N^{1/2},N^{1/2}\right]}\left(k\right)\left[\frac{1-b\left(k-k_3\right)}{k_3}v^b_{-k_3}u^b_{k}u^b_{k_3} - \frac{1-b\left(k+k_3\right)}{k_3}v^b_{k_3}u^b_{k}u^b_{-k_3}\right] \\
&=\sum_{k_3 \gtrsim N}1_{\left[-N^{1/2},N^{1/2}\right]}\left(k\right)\frac{b\left(k+k_3\right)-b\left(k-k_3\right)}{k_3}v^b_{k}u^b_{-k_3}u^b_{k_3} \\ 
&+\sum_{k_3 \gtrsim N}1_{\left[-N^{1/2},N^{1/2}\right]}\left(k\right)\left[\frac{b\left(k+k_3\right)-b\left(k-k_3\right)}{k_3}v^b_{k_3}u^b_{k}u^b_{-k_3} + \left(1-b\left(k-k_3\right)\right)e_{k_3}u_k^b\right].
\end{aligned}
\end{equation*}
By the fact that $b\left(k\right)$ is even and the mean value theorem, we have
\begin{equation}\label{eq:mean value thm of B(k)}
\left|b\left(k+k_3\right)-b\left(k-k_3\right)\right|= \left|b\left(k+k_3\right)-b\left(k_3-k\right)\right| = \MC{O}\left(\frac{\left|k\right|}{N}\right) = \MC{O}\left(N^{-\sigma}\right),
\end{equation}
for $\left|k\right| \lesssim N^{1/2}$. In addition, we have
\begin{equation}\label{eq:last bound1}
\left|\sum_{\left|k_3\right| \gtrsim N} \frac{1}{k_3} u_{-k_3}v_{k_3}\right| \le \left\|u\right\|_{H^{-1/2}_0}\left\|v\right\|_{H^{-1/2}_0},
\end{equation}
and 
\begin{equation}\label{eq:last bound2}
\left\|u\right\|_{L^{\infty}_tH^{-1/2}_x} \lesssim \left\|u\right\|_{Y^{-1/2}},
\end{equation}
by the Cauchy-Schwarz inequality and the global well-posedness. Thus, (\ref{eq:goal of MVT}) is proved by \eqref{eq:error sum}-\eqref{eq:last bound2}. From (\ref{eq:space embedding}), we conclude the $Z^{-1/2}$-term of (\ref{eq:N_3 est.}) is ${ \scriptstyle\MC{O}}_N\left(1\right)$. \\
\\
{\bf Step 3}. \\ 
To complete the proof of Theorem \ref{thm:est. of truncated flow}, we need to use a rescaling argument. Our claim is to show that \eqref{eq:est. of truncated flow} is true on the time interval $\left[0,T\right]$. In fact, this claim is equivalent to show that the $\alpha$-scaled problem with solution 
\begin{equation}\label{scaling}
u_\alpha\left(x,t\right) = \alpha^{-2}u\left(\frac{x}{\alpha},\frac{t}{\alpha^3}\right)
\end{equation}
is true on an interval $\left[0,\alpha^3T\right]$ and on domain $ \alpha\MB{T} $. Roughly speaking, if we can show that
$\left\|P_{ \le N^{1/2}}u_{\alpha}\right\|_{Y^{-1/2}}$, $\left\|P_{ \le N^{1/2}}v_{\alpha}\right\|_{Y^{-1/2}}$, $\left\|P_{ \le N^{1/2}}u^b_{\alpha}\right\|_{Y^{-1/2}}$ and $\left\|P_{ \le N^{1/2}}v^b_{\alpha}\right\|_{Y^{-1/2}}$ are sufficiently small in $\alpha \MB{T} $, then we are done by putting all previous step together.\\
\\ 
Although  the implicit constants depend on $\alpha$, we can obtain (\ref{eq:B_2 est.}) and (\ref{eq:N_3 est.}) on $\alpha \MB{T}$ by Remark \ref{rmk:rescaling est.}. More precisely,
by Step 1, Step 2 and Remark \ref{rmk:rescaling est.},
we have
\begin{equation}\label{eq: full est. rescaled}
\begin{split}
&\left\|P_{\le N^{1/2}}\left(u_{\alpha}-u_{\alpha}^b\right)\right\|_{Y^{-1/2}\left(\alpha \MB{T}\right)}\\
&\lesssim \alpha^{0+}\left[ \left\|P_{ \le N^{1/2}}v_{\alpha}\right\|_{Y^{-1/2}\left(\alpha \MB{T}\right)} \left\|P_{ \le N^{1/2}}\left(v_{\alpha} - v_{\alpha}^b\right)\right\|_{Y^{-1/2}\left(\alpha \MB{T}\right)} \right.\\
&+ \left\|P_{ \le N^{1/2}}v_{\alpha}^b\right\|_{Y^{-1/2}\left(\alpha \MB{T}\right)} \left\|P_{ \le N^{1/2}}\left(v_{\alpha} - v_{\alpha}^b \right)\right\|_{Y^{-1/2}\left(\alpha \MB{T}\right)} \\
&+ \left\|P_{ \le N^{1/2}}u_{\alpha}\right\|_{Y^{-1/2}\left(\alpha \MB{T}\right)}\left\|P_{ \le N^{1/2}}v_{\alpha}\right\|_{Y^{-1/2}\left(\alpha \MB{T}\right)} \left\|P_{ \le N^{1/2}}\left(v_{\alpha} - v_{\alpha}^b\right)\right\|_{Y^{-1/2}\left(\alpha \MB{T}\right)} \\
&+  \left\|P_{ \le N^{1/2}}u_{\alpha}\right\|_{Y^{-1/2}\left(\alpha \MB{T}\right)}\left\|P_{ \le N^{1/2}}v_{\alpha}^b\right\|_{Y^{-1/2}\left(\alpha \MB{T}\right)} \left\|P_{ \le N^{1/2}}\left(v_{\alpha} - v_{\alpha}^b\right)\right\|_{Y^{-1/2}\left(\alpha \MB{T}\right)} \\
&+ \left.\left\|P_{ \le N^{1/2}}v_{\alpha}^b\right\|_{Y^{-1/2}\left(\alpha \MB{T}\right)}\left\|P_{ \le N^{1/2}}v_{\alpha}^b\right\|_{Y^{-1/2}\left(\alpha \MB{T}\right)} \left\|P_{ \le N^{1/2}}\left(u_{\alpha} - u_{\alpha}^b\right)\right\|_{Y^{-1/2}\left(\alpha \MB{T}\right)}\right] \\
&+{ \scriptstyle\MC{O}}_N\left(1\right).
\end{split}
\end{equation}

Similarly, we also have the estimate with respect to $\left\|P_{\le N^{1/2}}\left(v_{\alpha}-v_{\alpha}^b\right)\right\|_{Y^{-1/2}\left(\alpha \MB{T}\right)}$. 
Hence, if we can prove 
\begin{equation}\label{eq: smallness of rescaled}
\left\|P_{\le N^{1/2}}u^{(b)}_{\alpha}\left(t,x\right)\right\|_{Y^{-1/2}\left(\alpha \MB{T}\right)} + \left\|P_{\le N^{1/2}}v^{(b)}_{\alpha}\left(t,x\right)\right\|_{Y^{-1/2}\left(\alpha \MB{T}\right)} \ll 1,
\end{equation}
for $N > N_0(T, \varepsilon, \left\|u_{0}\right\|_{H^{-1/2}_0\left(\MB{T}\right)}, \left\|v_{0}\right\|_{H^{-1/2}_0\left(\MB{T}\right),})$, then all terms of the right hand side of \eqref{eq: full est. rescaled} except for ${\scriptstyle\MC{O}}_N\left(1\right)$ are absorbed in the left hand side. Now we show that \eqref{eq: smallness of rescaled}. By the global bound of solutions for $ t\in [0,T]$, and scaling back from \eqref{scaling},
\begin{equation*}
\left\|P_{\le N^{1/2}}u_{\alpha}\left(t,x\right)\right\|_{Y^{-1/2}\left(\alpha \MB{T}\right)} \lesssim \left\|P_{\le N^{1/2}}u_{{\alpha},0}\left(x\right)\right\|_{{H^{-1/2}_0}\left(\alpha \MB{T}\right)} = \alpha^{-1}\| P_{\le \alpha N^{1/2}} u_0 \|_{H_0^{-1/2}(\MB{T})}.
\end{equation*}
We first choose $\alpha$ sufficiently large such that terms involving the difference are absorbed to the left hand side. In the estimates of remainder terms in the previous step, $ {\scriptstyle\MC{O}}_N\left(1\right) $ depends on $\alpha $, too. But after fixing $\alpha $, we choose $N$ sufficiently large so that $ {\scriptstyle\MC{O}}_N\left(1\right) $ to be small. We handle $v_{\alpha}$, $u^b_{\alpha}$ and $v^b_{\alpha}$ similarly. Consequently, we finish the proof of Theorem~\ref{thm:est. of truncated flow} due to $Y^{s} \subset C_{t}H^s$.

\subsection{Proof of Theorem \ref{thm:est. of different data}} $ $\\
The argument is highly similar to \cite{Colliander:2005vv} to prove Theorem \ref{thm:est. of different data}. We mainly prove the following local-in time estimate. 
\begin{lem}\label{lem:local est. of different data}
Let $N' \ge 1$, $\left(u_0,v_0\right) \in {H^{{-\frac{1}{2}}}_0} \times  {H^{{-\frac{1}{2}}}_0}$ and $\left(u'_0,v'_0\right) \in {H^{{-\frac{1}{2}}}_0} \times  {H^{{-\frac{1}{2}}}_0}$ such that $P_{\le N'}\left(u_0,v_0\right) = P_{\le N'}\left(u'_0,v'_0\right)$. If $T'$ is sufficiently small depending on $\left\|u_0\right\|_{H^{-1/2}_0}$, $\left\|u'_0\right\|_{H^{-1/2}_0}$, $\left\|v_0\right\|_{H^{-1/2}_0}$, and $\left\|v'_0\right\|_{H^{-1/2}_0}$, then we have
\begin{align*}
&\sup_{\left|t\right|\le T'}\left\|P_{ \le N'-(N')^{1/2}}\left(S_{CKdV}\left(t\right) u_0-S_{CKdV} \left(t\right)u'_0\right)\right\|_{{H^{-1/2}_0}} \\
&+\sup_{\left|t\right|\le T'} \left\|P_{ \le N'-(N')^{1/2}}\left(S_{CKdV}\left(t\right)v_0-S_{CKdV}\left(t\right)v'_0\right)\right\|_{{H^{-1/2}_0}}\\
&\le C\left(\left\|u_0\right\|_{H^{-1/2}_0}, \left\|u'_0\right\|_{H^{-1/2}_0}, \left\|v_0\right\|_{H^{-1/2}_0}, \left\|v'_0\right\|_{H^{-1/2}_0}\right) { \scriptstyle\MC{O}}_{N'}\left(1\right).
\end{align*}
\end{lem}
Theorem \ref{thm:est. of different data} can be proved by using Lemma~\ref{lem:local est. of different data}.  Roughly speaking, we divide the given time interval $\left[-T,T\right]$ into intervals which has length $\left|T'\right|$, and use repeatedly Lemma~\ref{lem:local est. of different data}. For this argument, we refer to Section 5 in \cite{Colliander:2005vv} for details.
\\
\begin{proof}[Proof of Lemma \ref{lem:local est. of different data}]
We only consider the difference between $u$ and $u'$ as $v$ and $v'$ case is handled in the same way. From the local well-posedness for \eqref{eq:Coupled KdV system}, we have the local estimates
\begin{equation}\label{eq:small}
\left\|u\right\|_{Y^{-1/2}} + \left\|u'\right\|_{Y^{-1/2}} \lesssim C ~\text{and}~\left\|v\right\|_{Y^{-1/2}} + \left\|v'\right\|_{Y^{-1/2}} \lesssim C,
\end{equation}
by choosing the sufficiently small time $T'$ depending on the $H_0^{-1/2}$-norms of $u_0$, $u'_0$, $v_0$ and $v'_0$. We apply $P_{\le M}$ in \eqref{eq:diff by parts original flow} to get,
\begin{equation}\label{eq:low freq eq}
\begin{aligned}
\partial_t 1_{\left[-M,M\right]}\left(k\right) {\bf u}_k &=  \partial_t\left\{\frac{1_{\left[-M,M\right]}\left(k\right)}{6} \sum_{k_1+k_2=k} \frac{e^{-i\phi\left({\rm k}\right)t}}{k_1k_2}{\bf v}_{k_1}{\bf v}_{k_2}\right\}  \\
&+2i~\frac{1_{\left[-M,M\right]}\left(k\right)}{6} \sum_{\substack{k_1+k_2+k_3=k \\ k_1+k_2 \not = 0}}1_{[-M,M]}\left(k_1+k_2\right) \frac{e^{-i\Phi\left({\rm k}\right)t}}{k_3}{\bf u}_{k_1}{\bf v}_{k_2}{\bf v}_{k_3}.
\end{aligned}
\end{equation}
Taking linear propagator back, we can rewrite the right hand side of \eqref{eq:low freq eq} as follows:
\begin{equation*}
\begin{aligned}
&\partial_t \MC{F}_x\left[P_{\le M}B'_2\left({\bf v},{\bf v}\right)\right]+\MC{F}_x\left[ P_{\le M} N'_3\left({\bf u},{\bf v},{\bf v}\right)\right] \\
&=\partial_t\left\{\frac{1_{\left[-M,M\right]}\left(k\right)}{6} \sum_{k_1+k_2=k} \frac{e^{-i\phi\left({\rm k}\right)t}}{k_1k_2}e^{-it\left(k_1^3+k_2^3\right)} v_{k_1}v_{k_2}\right\}  \\
&+2i~\frac{1_{\left[-M,M\right]}\left(k\right)}{6} \sum_{\substack{k_1+k_2+k_3=k \\ k_1+k_2 \not = 0}}1_{[-M,M]}\left(k_1+k_2\right) \frac{e^{-i\Phi\left({\rm k}\right)t}}{k_3}e^{-it\left(k_1^3+k_2^3+k_3^3\right)}u_{k_1}v_{k_2}v_{k_3} \\
&=:\partial_t \MC{F}_x\left[P_{\le M}B'_2\left(v,v\right)\right]+ \MC{F}_x\left[P_{\le M} N'_3\left(u,v,v\right)\right].
\end{aligned}
\end{equation*}
The constant $M$ will be the low frequency cut-off and will be chosen later. In order to show $  \|P_{N'-(N')^{1/2}}u(t) -u'(t) \|_{Y^{-1/2}}$ to be small for a short time $T'$, we analyse the nonlinear terms and show that the contribution from high frequency pieces is small and so regarded as remainder terms in ${ \scriptstyle\MC{O}}_{N'}\left(1\right) $. In addition, for the contribution from all low frequency pieces, we use the local stability theory. \\
\\
We first consider the trilinear term $N'_3(u,v,v)$. This part is also similar to Step 1 of Subsection \ref{sec: pf of thm trun}. In order to control $N'_3(u,v,v)$, we define the $(\mbox{\emph{error terms}})_3$ which has the $Z^{-\frac{1}{2}}$-norm of $ { \scriptstyle\MC{O}}_{N'}\left(1\right)$. The solutions $u$, $u'$, $v$ and $v'$ are decomposed into the three pieces using the following argument. By the global well-posedness and the pigeon-hole principle, we may find an interval $\left[M,M+(N')^{1/4}\right] \subseteq [N'-\left(N'\right)^{\frac12},N']$\footnote{ The constant $M$ shall be different from that in Subsection \ref{sec: pf of thm trun}.} such that 
\begin{equation}\label{eq:sub-small}
\begin{split}
&\left\|(P_{\le M+\left(N'\right)^{1/4}}-P_{\le M})u\right\|_{Y^{-1/2}} + \left\|(P_{\le M+\left(N'\right)^{1/4}}-P_{\le M})u'\right\|_{Y^{-1/2}}  \\
&+\left\|(P_{\le M+\left(N'\right)^{1/4}}-P_{\le M})v\right\|_{Y^{-1/2}} + \left\|(P_{\le M+\left(N'\right)^{1/4}}-P_{\le M})v'\right\|_{Y^{-1/2}} \lesssim \left(N'\right)^{-\sigma}.
\end{split}
\end{equation}
We fix such $M$ with $ N'-N'^{1/2} \le M \le N'$ and decompose $u$ as
\begin{equation*}
u = u_{low} + u_{med}+u_{hi},
\end{equation*}
where 
\begin{equation*}
u_{low} := P_{\le M}u, \hspace{1em} u_{med} := (P_{\le M+\left(N'\right)^{1/4}}-P_{\le M})u, \hspace{1em} u_{hi}:= (1-P_{\le M+\left(N'\right)^{1/4}})u.
\end{equation*}
From \eqref{eq:small} and \eqref{eq:sub-small}, we have
\begin{equation}\label{eq:split}
\left\|u_{low}\right\|_{Y^{-1/2}},\left\|u_{hi}\right\|_{Y^{-1/2}} \le C \hspace{1em} \mbox{and} \hspace{1em} \left\|u_{med}\right\|_{Y^{-1/2}} \lesssim \left(N'\right)^{-\sigma}.
\end{equation}
We also do the same decomposition for $u'$, $v$ and $v'$, and obtain analogous estimates like \eqref{eq:split}. Moreover, we denote 
\begin{equation*}
P_{\le M}N'_3(u, v, v) = P_{\le M}N'_3(u_{low},v_{low}, v_{low}) + \text{(remainder terms)}_3.
\end{equation*}
First of all, from Lemma \ref{lem:tri resonant est} and \ref{lem:local tri. est.}, any term in $\text{(remainder terms)}_3$ involving $u_{med}$, $u'_{med}$, $v_{med}$ or $v'_{med}$ is $\MC{O}\left(\left(N'\right)^{-\sigma}\right)$ in $(\mbox{\emph{error terms}})_3$. We now consider terms which involve in $v_{hi}$. As before, $\text{(remainder terms)}_3$ is split into the resonant case and the nonresonant case. From \eqref{eq: low-hi-hi tri}, the typical term of resonant case is $P_{\le M}N'_3(u_{low}, v_{hi},  v_{hi})$ and therefore, we have 
\begin{equation*}
\left\|P_{\le M}N'_3(u_{low}, v_{hi},  v_{hi})\right\|_{Z^{-1/2}} \sim { \scriptstyle\MC{O}}_{N'}\left(1\right),
\end{equation*}
by \eqref{eq: high freq resonant est.}. For the nonresonant case, we estimate as Step 1 of Subsection \ref{sec: pf of thm trun}. Since we have $n_3 \gtrsim (N')^{1/4}$, Lemma \ref{lem:local tri. est.}, and the estimate of \eqref{eq: result of trun. trilinear} in Subsection~\ref{sec: pf of thm trun}, the nonresonant case of $\text{(remainder terms)}_3$ is bounded by $\MC{O}\left(\left(N'\right)^{-\sigma}\right)$. In other words, $P_{\le M}N'_3(u,v,v)$ can be written $P_{\le M}N'_3(u_{low}, v_{low},v_{low}) + (\mbox{\emph{error terms}})_3$ with $ \| (\mbox{\emph{error terms}})_3 \|_{Z^{-1/2}} =  {\scriptstyle\MC{O}}_{N'}\left(1\right) $. Due to $P_{\le M}e^{\pm t \partial^3_x} = e^{\pm t\partial^3_x}P_{\le M}$, we have
\begin{equation*}
P_{\le M}N'_3(u,v,v) = P_{\le M}N'_3({\bf u}_{low}, {\bf v}_{low}, {\bf v}_{low}) + (\mbox{\emph{error terms}})_3,
\end{equation*}
where ${\bf u}_{low} = P_{\le M}{\bf u}$ and ${\bf v}_{low} = P_{\le M}{\bf v}$.\\
\\
The bilinear term $B'_2(v,v)$ can be analysed in a similar way. Let $(\mbox{\emph{error terms}})_4$ be term which has the $Y^{-\frac{1}{2}}$-norm of ${ \scriptstyle\MC{O}}_{N'}\left(1\right)$. We choose the same constant $M$ in  $N_3'(u,v,v)$-term case and split the solution $v$ and $v'$ into as follows,
\begin{equation*}
v = v_{low} + v_{hi},
\end{equation*}
where 
\begin{equation*}
v_{low} = P_{ \le M}v  \hspace{1em} \text{and} \hspace{1em} v_{hi} = \left(1-P_{ \le M}\right)v.
\end{equation*}
We denote 
\begin{equation}\label{eq:truncated bilinear}
P_{\le M}B'_2(v, v) = P_{\le M}B'_2(v_{low}, v_{low}) + \text{(remainder terms)}_4.
\end{equation}
From \eqref{eq:truncated bilinear}, $\text{(remainder terms)}_4$ has $v_{hi}$ term only. Hence, $\text{(remainder terms)}_4$ is bounded by ${ \scriptstyle\MC{O}}_{N'}\left(1\right)$ from \eqref{eq: result of bi est.} and $M\in [N'-\left(N'\right)^{\frac12},N']$. Therefore, the bilinear term $B'_2(v, v)$ can be written $B'_2(v_{low}, v_{low}) + (\mbox{\emph{error terms}})_4$, and so 
\begin{equation*}
B'_2(v, v) = B'_2({\bf v}_{low}, {\bf v}_{low}) + (\mbox{\emph{error terms}})_4,
\end{equation*}
by $P_{\le M}e^{\pm t \partial^3_x} = e^{\pm t\partial^3_x}P_{\le M}$.
\\
\\
Consequently, ${\bf u}_{low}$ obeys the equation,
\begin{equation}\label{eq:low equation1}
\partial_t{\bf u}_{low} = \partial_t P_{\le M}B'_2({\bf v}_{low}, {\bf v}_{low})+P_{\le M}N'_3({\bf u}_{low},{\bf v}_{low},{\bf v}_{low}) + \sum_{i=3,4}(\mbox{\emph{error terms}})_{i}.
\end{equation} 
 In the same manner, the function ${\bf u}'_{low}$ obeys the equation
\begin{equation}\label{eq:low equation2}
\partial_t{\bf u}'_{low} = \partial_t P_{\le M}B'_2({\bf v}'_{low}, {\bf v}'_{low})+P_{\le M}N_3'({\bf u}'_{low},{\bf v}'_{low},{\bf v}'_{low}) + \sum_{i=3,4}(\mbox{\emph{error terms}})_{i}.
\end{equation}
From the local well-posedness for (\ref{eq:low equation1}) or \eqref{eq:low equation2}, transforming back, $u_{low}\left(0\right) = u'_{low}\left(0\right)$, and the fact that by the rescaling argument as in Subsection \ref{sec: pf of thm trun}, we may assume that the initial data are small in $Y^{-1/2}$, we have
\begin{equation*}
\left\|u_{low} - u'_{low}\right\|_{Y^{-1/2}} \lesssim  { \scriptstyle\MC{O}}_{N'}\left(1\right).
\end{equation*}
We also get 
\begin{equation*}
\left\|v_{low} - v'_{low}\right\|_{Y^{-1/2}} \lesssim  { \scriptstyle\MC{O}}_{N'}\left(1\right),
\end{equation*}
by the similar argument, and we thus finish the proof by $Y^{s} \subset C_{t}H^s$.
\end{proof}

\begin{rmk}
Although we provide the proof for \eqref{eq:Majda-Biello system}, the same proof works for a more general case \eqref{eq:Diagonalize Coupled KdV system} if there is a  global control of solutions on $C_tH_0^{-\frac{1}{2}}([0,T]\times\MB{T})$.
\end{rmk}

\bibliography{nonsqueezing_abbr}

\end{document}